\documentclass[12pt,a4paper]{amsart}
\usepackage{txfonts}

\usepackage{hyperref}
\usepackage{latexsym}
\usepackage{amssymb}

\newtheorem{theorem}{Theorem}[section]
\newtheorem{lemma}[theorem]{Lemma}
\newtheorem{corollary}[theorem]{Corollary}

\theoremstyle{definition}

\newtheorem{remark}[theorem]{Remark}

\numberwithin{equation}{section}

\begin{document}

\title[ On some determinant and matrix inequalities with a geometrical flavour ]{{\bf On some determinant and matrix inequalities with a geometrical flavour}}
\author{\small TING CHEN}

\address{Ting Chen:   School of Mathematics, University of Edinburgh, EH9 3JZ, UK.}
 \email{t.chen-16@sms.ed.ac.uk}

\footnotetext{\hspace{-3pt}
}

\begin{abstract}
In this paper we study some determinant inequalities and matrix inequalities which have a geometrical flavour. 
We first examine some inequalities which place work of Macbeath \cite{Macbeath} in a more general setting and also relate to recent work of Gressman \cite{Gressman}. 
In particular, we establish optimisers for these determinant inequalities. 
We then use these inequalities to establish our main theorem which gives a geometric inequality of matrix type 
which improves and extends some inequalities of Christ in \cite{Christ}.

\end{abstract}

\maketitle

\section{Introduction}

\subsection{Notation and Preliminaries}

Let $\mathbb{R}^{n}$ be the $n$-dimensional Euclidean space, $n \geq 1$.
$|\cdot|$ denotes the Lebesgue measure on  $\mathbb{R}^{n}$ and
the absolute value on $\mathbb{R}$.
Denote  $\mathfrak{M}^{n \times n}(\mathbb{R})$ by a set of all $n \times n$ real matrices.
Let $B (0, r)$ be the ball centred at $0$ with radius $r$.
For $A\subset \mathbb{R}^{n} $ of finite Lebesgue measure, we define the symmetric rearrangement of 
$A$ as
\begin{center}
$A^{\ast}:=\{x: |x|<r \} \equiv B (0, r)$,  \  with $|A^{\ast}| =|A|$.
\end{center}
That is, $v_{n} r^{n}=|A|$, where $v_{n}$ is the volume of unit ball in $\mathbb{R}^{n}$.
We then define the symmetric decreasing rearrangement of a  nonnegative measurable function $f$ as
$$f^{\ast } (x):= \int_{0}^{\infty} \chi_{\{f>t\}^{\ast}} (x) dt, $$
where $ \chi_{\{f>t\}}$ is the characteristic function of the level set $ \{x: f(x)>t \}$,
and define the Steiner symmetrisation of  $f$ with respect to the $j$-th coordinate as
$$\mathcal{R}_{j}f(x_{1}, \dots,x_{n})=f^{\ast j}(x_{1}, \dots,x_{n}) :=  \int_{0}^{\infty} \chi_{\{f(x_{1}, \dots ,x_{j-1}, \cdot, x_{j+1}, \dots , x_{n})>t\}^{\ast}}(x_{j}) dt.$$

Let $u\in \mathbb{R}^{n}$ be a unit vector, $u^{\perp}$ be its orthogonal complement.
Then for any $x\in \mathbb{R}^{n}$, it can be uniquely written as $x=tu+y$ where $y\in u^{\perp}$.
We define the Steiner symmetrisation of $A$ with respect to the direction $u$ as
$$\mathcal{S}_{u}(A):= \{tu+y: A\cap(\mathbb{R}u+y)\neq \phi, |t| \leq \frac{|A\cap(\mathbb{R}u+y)|}{2}\}.$$
Obviously, $\mathcal{R}_{j}\chi_{A}$ is the Steiner symmetrisation of $A$ with respect to the direction $e_{j}$, $1\leq j \leq n$.
For simplicity, we denote  $\mathcal{S}_{e_{n}} \mathcal{S}_{e_{n-1}} \dots  \mathcal{S}_{e_{1}} (E)$  by $\mathcal{S}E$,
where
$\{e_{1}, \dots, e_{n}\}$ is the standard orthonormal basis in  $\mathbb{R}^{n}$.

\medskip

One easily sees that for any measurable set $E \subset \mathbb{R}^{n}$
$$\sup\limits_{x \in E^{\ast}}  \ |x| \leq \displaystyle{\sup_{x \in E }} \ |x|,   \eqno (1.1)$$
and from this it is not hard to see that  
$$\sup\limits_{x,y \in E^{\ast}}  \ |x-y| \leq \displaystyle{\sup_{x,y \in E }} \ |x-y|.   \eqno (1.2)$$
One way to obtain this  is as follows.
$$\sup\limits_{x,y \in E} \ |x-y|= \sup\limits_{z \in E-E} \ |z| \geq \sup\limits_{z \in (E-E)^{\ast}} \ |z|. \eqno (1.3)$$
For any $A,B \in \mathbb{R}^{n}$ of finite Lebesgue measure, it follows from the Brunn-Minkowski inequality that
$$A^{\ast}+B^{\ast} \subset (A+B)^{\ast}. \eqno (1.4)$$
Applying (1.4) in (1.3) implies
$$\sup\limits_{x,y \in E} \ |x-y|
\geq \sup\limits_{z\in (E-E)^{\ast}} \ |z|
\geq \sup\limits_{x\in E^{\ast}, y\in E^{\ast}} \ |x-y|,$$
which completes (1.2).

Let $E$ be a measurable set of finite volume in $\mathbb{R}^{n}$.
By the definition of the symmetric rearrangement,
\begin{center}
$E^{\ast}=B(0,r)$,  \ with  \ $v_{n} r^{n} = |E|$.
\end{center}
Clearly,
$$\sup\limits_{x\in E^{\ast}} \ |x|=r, \ \sup\limits_{x,y \in E^{\ast}} \ |x-y| =2r.$$
By (1.1) and (1.2) we have the following sharp inequality
$$|E| \leq  v_{n} \sup\limits_{x\in E} \ |x|^{n}, \eqno (1.5)$$
$$|E| \leq \frac{v_{n}}{2^{n}} \sup\limits_{x,y \in E} \ |x-y|^{n}. \eqno (1.6)$$
Moreover,  optimisers of both  (1.5) and (1.6) are balls in $\mathbb{R}^{n}$.
Inequality (1.6) is an isodiametric inequality, that is, amongst all sets with given diameter the ball has maximal volume.

\subsection{Macbeath's inequalities}

We now  go on to study the analogues of  (1.5) and (1.6) where we replace the distance norm by a volume or determinant,
so the question  becomes that of studying inequalities of the form
$$|E| \leq A_{n} \sup\limits_{\substack{y_{j} \in E \\ j=1, \dots, n}}  \det(0, y_{1}, \dots, y_{n}),   \eqno (1.7)$$
and
$$|E| \leq B_{n} \sup\limits_{\substack{y_{j} \in E \\ j=1, \dots, n+1}} \det(y_{1}, \dots, y_{n+1}),   \eqno (1.8)$$
which are supposed to hold for any measurable set $E$ in $\mathbb{R}^{n}$.
Here 
$$\det(y_{1}, \dots, y_{n+1}):=n!  \mathrm{vol}  (\mathrm{co}\{y_{1}, \dots, y_{n+1}\}).$$
So $\det(y_{1}, \dots, y_{n+1}) \geq 0$.
The precise value of $\det(y_{1}, \dots, y_{n+1})$ is the absolute value of the determinant
of the matrix $(y_{1}-y_{n+1}, \dots, y_{n}-y_{n+1})_{n\times n}$. 
In the special case when $n=1$, they become of the type (1.5) and (1.6) automatically.
Note that both (1.7) and (1.8) are $\mathrm{GL}_{n}(\mathbb{R})$  invariant, and (1.8) is translation invariant while (1.7) is not. 
Actually, it is enough to study  convex measurable sets in $\mathbb{R}^{n}$, since 
$$\sup\limits_{\substack{y_{j} \in E \\ j=1, \dots, n}}  \det(0, y_{1}, \dots, y_{n})
=\sup\limits_{\substack{y_{j} \in \mathrm{co}(E) \\ j=1, \dots, n}}  \det(0, y_{1}, \dots, y_{n}),$$
and
$$\sup\limits_{\substack{y_{j} \in E \\ j=1, \dots, n+1}} \det(y_{1}, \dots, y_{n+1})
=\sup\limits_{\substack{y_{j} \in \mathrm{co}(E) \\ j=1, \dots, n+1}} \det(y_{1}, \dots, y_{n+1}).$$

\medskip

We are interested in  the best constants  $A_n$, $B_n$ and  their optimsers.
It is not hard to deduce that  the best constant $A_n$ and  $B_n$ are related by
$$B_n \leq A_n \leq (n+1)B_n.   \eqno (1.9) $$
Indeed, the  translation invariance of (1.8) allows us to  assume that $0 \in E$.
Then  $B_n \leq A_n$ follows immediately. On the other hand, 
by the basic  determinant property we have
$$\det(y_{1}, \dots, y_{n+1})
\leq  \sum\limits_{j=1}^{n+1} \det(0, y_{1}, \dots,  y_{j-1}, y_{j+1}, \dots,  y_{n}),$$
which implies that
$$\sup\limits_{\substack{y_{j} \in E \\ j=1, \dots, n+1}} \det(y_{1}, \dots, y_{n+1})
\leq (n+1) \sup\limits_{\substack{y_{j} \in E \\ j=1, \dots, n}}  \det(0, y_{1}, \dots, y_{n}).$$
That completes  $A_n \leq (n+1)B_n$.
So in the special case when $n=1$, 
we have $A_1=2$,  $B_1= 1$ that  follows from  (1.5) and (1.6).

\medskip

Geometrically, the right side of (1.8) relates to  the maximal volume of $n$-simplex whose vertices are in $E$.
The relationship between the maximal volume of the $n$-simplex whose vertices are in $E$ and the measure of $E$ has been studied before (see \cite{Kanazawa}, \cite{Macbeath}).
It is well known that by compactness given a compact convex set $E\subset \mathbb{R}^{n}$, there exists a simplex $T \subset E$ of maximal volume. 
Let $F$ be a facet of $T$, $v$ the opposite vertex, and $H$ the hyperplane through $v$ parallel to $F$.
Then $H$ supports $E$, since otherwise one would obtain a contradiction to the maximality of the volume of $T$.
Since $F$ is an arbitrary facet of $T$, $T$ is contained in the simplex  $-n(T-c)+c$, where $c$ is the centroid of $T$. See \cite{Kanazawa} for details.
So $T \subset E \subset -n(T-c)+c$, and thus
$$ |E| \leq n^{n}|T|.   \eqno(1.10)$$
which implies that 
$$B_n \leq n^{n},  \  A_n \leq (n+1)n^{n}.$$

In 1950, Macbeath \cite{Macbeath} already gave the sharp version of (1.10) and (1.8) as follows.
Given a compact convex set $E \subset \mathbb{R}^{n}$, denote  $\mathfrak{B}_{m}$  the set of convex polytopes with  at most $m$ vertices in $E$,
and denote  $\mathfrak{C}_{m}$  the set of convex polytopes with at most $m$ vertices in $E^{\ast}$. Then
$$\sup\limits_{T^{\prime} \in \mathfrak{C}_{m}} |T^{\prime}| \leq \sup\limits_{T\in \mathfrak{B}_{m}} |T|. \eqno(1.11)$$
So when $m=n+1$, (1.11) gives 
$$\sup\limits_{\substack{y_{j} \in E^{\ast}  \\ j=1, \dots, n+1}} \det(y_{1}, \dots, y_{n+1})  
\leq  \sup\limits_{\substack{y_{j} \in E \\ j=1, \dots, n+1}} \det(y_{1}, \dots, y_{n+1}).$$
Moreover the problem is clearly affine invariant, thus the extremising sets turn out to be  balls and ellipsoids for (1.8).
Because the maximal simplex with vertices on a  ball  is the regular simplex with all sides equal, 
 we can obtain the corresponding best constant $B_n$. 
However, we do not believe that the sharp value of $A_n$ in (1.7)
has been given previously.

\subsection{Our Results}

In this paper we shall  give an alternative method to derive (1.7) and (1.8) with sharp constants $A_n$, $B_n$.
In Section 2, we will study some rearrangement inequalities which together with some work  in \cite{Chen} establish this.
A key ingredient will be Lemma 4.7 of \cite{Chen}, stating that 
for any $E_{j}\subset \mathbb{R}$ of finite Lebesgue measure, and $a_{j}\in \mathbb{R}$, $j=1, \dots, l$,
$$\sup\limits_{x_{j}\in E_{j}^{\ast}}  \ |\sum_{j=1}^{l} a_{j}x_{j}|  \leq \sup\limits_{x_{j}\in E_{j}} \ |\sum_{j=1}^{l} a_{j}x_{j}|. \eqno (1.12)$$
See Lemma 2.2 for the proof.

\medskip

More generally, returning to the inequalities (1.1), (1.2), we see there are functional versions.
One can consider a bilinear functional  rearrangement version of (1.2).
For all nonnegative measurable functions $f, g $ defined on $\mathbb{R}^{n}$,
$$\displaystyle{\sup_{x,y}} \  f^{\ast}(x) g^{\ast}(y) |x-y| \leq \displaystyle{\sup_{x,y}} \ f(x) g(y) |x-y|  \eqno (1.13)$$
holds.  Likewise, by the same argument as in its proof we also have
$$\displaystyle{\sup_{x}} \ f^{\ast}(x) |x| \leq \displaystyle{\sup_{x}} \ f(x) |x|. \eqno (1.14)$$
For the proof,  see Lemma 4.2 in \cite{Chen}.

\medskip

In Section 2,    generalizing them we arrive at the following multilinear functional  rearrangement inequalities,
$$\displaystyle{\sup_{y_{j}}} \prod_{j=1}^{n} f_{j}^{\ast}(y_{j})  \det(0, y_{1}, \dots, y_{n})
\leq   \displaystyle{\sup_{y_{j}}}  \prod_{j=1}^{n} f_{j}(y_{j})  \det(0, y_{1}, \dots, y_{n}),  \eqno (1.15)$$
and
$$\displaystyle{\sup_{y_{j}}} \prod_{j=1}^{n+1} f_{j}^{\ast}(y_{j})  \det(y_{1}, \dots, y_{n+1})
\leq   \displaystyle{\sup_{y_{j}}}  \prod_{j=1}^{n+1} f_{j}(y_{j})  \det(y_{1}, \dots, y_{n+1}),  \eqno (1.16)$$
which hold for any nonnegative measurable functions vanishing at infinity $f_{j} $ defined on $\mathbb{R}^{n}$,
in the sense that all its positive level sets have finite measure,
\begin{center}
$| \{x: |f(x)| >t \}| < \infty$, for all $t >0$.
\end{center}
As a matter of fact,  we  establish much more general inequalities in Theorem 2.5 below.
Then we get (1.7), (1.8) with the sharp constants by specialising to $f_{j}=\chi_{E}$ in (1.15)-(1.16), 
which also includes Macbeath's work (1.11) when $m=n+1$.

\medskip

There is another class of inequalities  concerning analogues  of (1.5), (1.6) where we  replace
the underlying  Euclidean space   $\mathbb{R}^{n}$  by the space of $n \times n$  real  matrices, and
the Euclidean norm by  $| \det(A) |$.
For example, Christ first  studied this type of inequality  in  \cite{Christ}.
Here ``$\det$"  becomes ordinary determinant of a matrix.

\medskip

\hspace{-13pt}{\bf Sublemma 14.1.}\cite{Christ} {\it
For any $n \geq 1$ there exists $C \in \mathbb{R}^{+}$ with the following property.
Let $E \subset \mathfrak{M}^{n \times n}(\mathbb{R})$ be a compact convex set satisfying $|E| < \infty$ and $E=-E$. Then
there exists $A\in E$  satisfying
$$ | \det(A) | \geq C |E|^{\frac{1}{n}},    \eqno (1.17)$$
where $|\cdot |$ denotes the Lebesgue measure on Euclidean space $\mathbb{R}^{n^{2}}$ and
the absolute value on $\mathbb{R}$.}

\medskip

\hspace{-13pt}{\bf Lemma 13.2.}\cite{Christ} {\it
For any $n \geq 1$ there exists $c, C \in \mathbb{R}^{+}$  and $k \in \mathbb{N}$ with the following property.
Let  $E$ be a measurable set in $\mathfrak{M}^{n \times n}(\mathbb{R})$  satisfying $|E| < \infty$. Then there exist
$T_{1}, \dots, T_{k} \in E$ and coefficients $s_{j} \in \mathbb{Z}$ satisfying
$|s_{j}| \leq c$, $\sum\limits_{j=1}^{k} s_{j}=0$, such that
$$| \det(\sum\limits_{j=1}^{k} s_{j}T_{j}) | \geq C |E|^{\frac{1}{n}}.     \eqno (1.18) $$}

\hspace{-13pt}{\bf Remarks 1.}

1. Let $\widetilde{E}=E-A := \{T-A: T \in E\}$ with $A \in \mathfrak{M}^{n \times n}(\mathbb{R})$,
then by Lemma 13.2 there exist $T_{1}, \dots,T_{k} \in E$
and  $s_{j} \in \mathbb{Z}$ satisfying $|s_{j}| \leq c$, $\sum\limits_{j=1}^{k} s_{j}=0$, such that
$$| \det(\sum\limits_{j=1}^{k} s_{j}(T_{j}-A)) |  = | \det(\sum\limits_{j=1}^{k} s_{j}T_{j}) | \geq C |E|^{\frac{1}{n}}=C |\widetilde{E}|^{\frac{1}{n}},  \eqno (1.19)$$
which shows (1.18) has  a translation invariance property that (1.17) lacks.

2. Based on the translation variance  property, we have an equivalent form of Lemma 13.2:
there exists $c, C \in \mathbb{R}^{+}$ such that for any  $E \subset \mathfrak{M}^{n \times n}(\mathbb{R})$
we can always select $T_{1}, \dots,T_{k} \in E$ and coefficients $s_{j} \in \mathbb{Z}$ satisfying $|s_{j}|  \leq c$, such that
$| \det(\sum\limits_{j=1}^{k} s_{j}T_{j}) | \geq C |E|^{\frac{1}{n}}$.

The equivalence is as follows. Supposing $A \in E$, denote $\widetilde{E}=E-A$.
Then if there exist $\overline{T}_{1}=T_{1}-A, \dots, \overline{T}_{k}=T_{k}-A \in \widetilde{E}$, where
$T_{j}\in E$, $1\leq j\leq k$,
and  there exist $s_{j} \in \mathbb{Z}$  satisfying $|s_{j}|   \leq c$, such that
$$| \det(\sum\limits_{j=1}^{k} s_{j}\overline{T}_{j}) | \geq C |\widetilde{E}|^{\frac{1}{n}}.$$
That is,
$$| \det( s_{1}T_{1}+ \dots + s_{k}T_{k} - (s_{1}+ \dots + s_{k}) A)| \geq C |\widetilde{E}|^{\frac{1}{n}}=C |E|^{\frac{1}{n}},$$
which satisfies the conditions of  Lemma 13.2.

More specifically,  when proving Lemma 13.2 Christ \cite{Christ} gave that under the same hypothesis of Lemma 13.2,
there exist  $A_{j} \in E$, and $s_{j} \in \{0, 1\}$, $j= 1, \dots, n$, such that
$$|\det(\sum\limits_{j=1}^{n} s_{j}A_{j}) | \geq  C |E|^{\frac{1}{n}} $$
which implies that for any measurable $E \subset \mathfrak{M}^{n \times n}$,
$$\sup\limits_{\substack{A_{1}, \dots, A_{n} \in E  \\  s_{1}, \dots, s_{n} \in \{0, 1\} }} |\det( s_{1}A_{1}+ \dots + s_{n}A_{n}) |
\gtrsim_{n} |E|^{\frac{1}{n}}. \eqno (1.20)$$

\medskip

In this paper we will improve (1.17)-(1.18) as follows, mainly  relying  on the rearrangement inequality (1.12).

\medskip

\hspace{-13pt}{\bf Main Theorem.}\ {\it
There exists a finite constant $\mathcal{C}_{n}$ such that  for any  measurable sets $E_j \subset \mathfrak{M}^{n \times n}(\mathbb{R})$ of finite measure, $j=1, \dots, n$,
$$ \prod\limits_{j=1}^{n}|E_j|^{\frac{1}{n^{2}}}  
\leq \mathcal{C}_{n} \sup\limits_{\substack{A_{j} \in E_{j}  \\  j=1, \dots, n}}  | \det(A_{1}+ \dots + A_{n} ) |.   \eqno (1.21)$$}

The main theorem implies (1.17) holds for all  compact convex sets  in  $\mathfrak{M}^{n \times n}(\mathbb{R})$ 
and extends  Lemma 13.2  as described below. In particular, we see from the main Theorem that all the $s_j$ in (1.20)
can be taken to be $1$.

\medskip

\hspace{-13pt}{\bf Corollary A.}\ {\it
There exists a finite constant $\mathcal{A}_{n}$ such that  for any  measurable set $E \subset \mathfrak{M}^{n \times n}(\mathbb{R})$ of finite measure,
for any non-zero scalar $\lambda_{j} \in \mathbb{R}$, $j=1, \dots, n$,
$$(\prod_{j=1}^{n} |\lambda_{j}|)  |E|^{\frac{1}{n}}  
\leq \mathcal{A}_{n} \displaystyle{\sup_{A_{j} \in E}}  \  | \det(\lambda_{1} A_{1}+ \dots + \lambda_{n} A_{n} )  |.  \eqno (1.22)$$}

\hspace{-13pt}{\bf Corollary B.}\ {\it
There exists a finite constant $\mathcal{B}_{n}$ such that  for any  measurable compact convex set  $E \subset \mathfrak{M}^{n \times n}(\mathbb{R})$ of finite measure,
$$|E|^{\frac{1}{n}}  \leq \mathcal{B}_{n} \displaystyle{\sup_{A \in E}}  \  | \det(A)  |.  \eqno (1.23)$$}
See Section 3 for the proof of Corollary B.

\medskip

\hspace{-13pt}{\bf Remarks 2.}

1. One can easily check that
$$ \sup\limits_{A \in \mathrm{co} \{0, E \}}  \  | \det(A)|=\sup\limits_{A \in E}  \  | \det(A)|.$$
This is because $ |\det( \lambda A)|= \lambda^{n}  |\det( A)|$  for any  $\lambda \in [0,1]$,
so we can always assume that $0 \in E$.
Given a measurable $E \subset \mathfrak{M}^{n \times n}(\mathbb{R})$, by scaling let $\widetilde{E}= r E$, $0 \neq r \in \mathbb{R}$, then
$$(| \widetilde{E}|)^{\frac{1}{n}}= (r^{n^{2}} |E|)^{\frac{1}{n}}=r^{n}|E|^{\frac{1}{n}},$$
and
$$\displaystyle{\sup_{A \in \widetilde{E}}}  \  | \det(A)|= r^{n} \displaystyle{\sup_{A \in E}}  \  | \det(A)|.$$
However,  (1.23) is not translation invariant.

2. We use a counterexample to show that (1.23) fails without the convex condition.
Take $n=2$  as an example, and let
\begin{center}
$E= \{(a,b,c,d): 0\leq ad \leq 1, 0 \leq bc \leq 1,  \  $and$ \ 1/N \leq a \leq N,  1/N \leq b \leq N\}$.
\end{center}
Then we have
$$\displaystyle{\sup_{A \in E}}  \  | \det(A)|
= \displaystyle{\sup_{A \in E}}  \   | \det
\left(
 \begin{array}{cc}
a & c \\
b & d \\
 \end{array}
 \right) |
\leq 2.$$
and $|E|= (2 \ln N)^{2}$.
Let $N \rightarrow \infty$, then we get the contradiction to (1.23).

\medskip

\hspace{-13pt}{\bf Remarks 3.}

1. An open  problem is what the best constants  $\mathcal{A}_{n}$,  $\mathcal{B}_{n}$,  $\mathcal{C}_{n}$ are.
We prove in this paper that balls or ellipsoids are not their optimisers. 

2. Note that  inequalities of matrix type introduced in this part do not enjoy an obvious affine invariance.
Nevertheless, there is an important action of $\mathrm{SL}_{n}(\mathbb{R})$ on  $\mathfrak{M}^{n \times n}(\mathbb{R})$
by premultiplication. That is, if $T\in \mathrm{GL}_{n}(\mathbb{R})$, $A\in \mathfrak{M}^{n \times n}(\mathbb{R})$ and
$E \subset  \mathfrak{M}^{n \times n}(\mathbb{R})$, then
$$\det(TA)=\det(T) \det(A)$$
and 
$$|TE|=|\det(T)|^{n} |E|.$$
So both matrix inequalities in this paper are invariant under premultiplication by a matrix of unimodular determinant.
We do not use the invariance of the entire problem under the action of left-multiplication by members of
$\mathrm{SL}_{n}(\mathbb{R})$ but instead the facts which underly this invariance, i.e. that this action preserves determinants of individual matrices 
and preserves volumes of sets. It enters as a ``catalyst"  in order to obtain a measure theoretic consequence and its presence vanishes without trace.

\bigskip

\section{Determinant inequalities}

In this section we study the determinant inequalities discussed in the introduction.
First we recall an estimate  by Gressman  \cite{Gressman} as follows.

\begin{lemma} \cite{Gressman}
There exists a finite constant  $C_{n}$ such that
for any $y \in \mathbb{R}^{n}$, for any measurable sets $E_{1}, \dots, E_{n}$ in $\mathbb{R}^{n}$ and for any $\delta >0$
$$| \{(y_{1}, \dots, y_{n})\in E_{1} \times \cdots \times E_{n}:
\det(y, y_{1}, \dots, y_{n}) < \delta \}| \leq C_{n}  \delta \displaystyle{\prod_{j=1}^{n}} |E_{j}|^{1-\frac{1}{n}}. \eqno (2.1)$$

\end{lemma}

As an immediate consequence of (2.1), we  obtain  the following inequality (2.2).
With  the same constant $C_n$, we have for any $y \in \mathbb{R}^{n}$,
for any  measurable sets   $E_{j} \subset \mathbb{R}^{n}$, $1 \leq j \leq n$,
$$ \displaystyle{\prod_{j=1}^{n}} \  |E_{j}|^{\frac{1}{n}} \leq C_{n}
\sup\limits_{y_{1} \in E_{1}, \dots, y_{n} \in E_{n} }  \det(y, y_{1}, \dots, y_{n}).  \eqno (2.2)$$
One way to see this  is as follows. Let $y \in \mathbb{R}^{n}$ and suppose
$$\sup\limits_{y_{1} \in E_{1}, \dots, y_{n} \in E_{n} } \det(y, y_{1}, \dots, y_{n})=s < \infty.$$
It follows from Lemma 2.1  that for all measurable sets $E_{j} \subset \mathbb{R}^{n}$, $1 \leq j \leq n$,
$$| \{(y_{1}, \dots, y_{n}) \in E_{1} \times \cdots \times E_{n}:  \det(y, y_{1}, \dots, y_{n}) \leq s  ) \} | 
\leq C_{n} s  \displaystyle{\prod_{j=1}^{n}}  |E_{j}|^{1-\frac{1}{n}}.$$
Note that $s = \sup\limits_{y_{1} \in E_{1}, \dots, y_{n} \in E_{n}}  \det(y, y_{1}, \dots, y_{n})$,  so
$$| \{(y_{1}, \dots, y_{n}) \in E_{1} \times \cdots \times E_{n}: \det(y, y_{1}, \dots, y_{n}) \leq s  ) \} |= \displaystyle{\prod_{j=1}^{n}} \ |E_{j}|.$$
Therefore,
$$\displaystyle{\prod_{j=1}^{n}} \ |E_{j}| \leq C_{n} s  \displaystyle{\prod_{j=1}^{n}}  |E_{j}|^{1-\frac{1}{n}}.$$
That is,
$$\displaystyle{\prod_{j=1}^{n}}  |E_{j}|^{\frac{1}{n}}  \leq C_{n} s=C_{n}  \sup\limits_{y_{1} \in E_{1}, \dots, y_{n} \in E_{n} } \det(y, y_{1}, \dots, y_{n}),$$
which completes (2.2).

\medskip

This motivates a multilinear perspective. Later on, we will prove the sharp version of (2.1)-(2.2).
More generally,  functional versions of (2.2) have been studied in  \cite{Chen}.
As  shown  in Theorem 3.1 of  \cite{Chen},
for any nonnegative measurable functions $ f_{j} \in  L^{p_{j}}(\mathbb{R}^{n}) $, 
$$  \displaystyle{\prod_{j=1}^{n+1}} \|f_{j}\|_{p_{j}}   \leq
C_{n, p_{j}} \displaystyle{\sup_{y_{j}}} \ \displaystyle{\prod_{j=1}^{n+1}} \ f_{j}(y_{j}) \det(y_{1}, \dots, y_{n+1})^{\gamma}   \eqno(2.3)$$
holds, if  and only if $p_{j}$ satisfy
$ \frac{1}{p_{j}} < \frac{\gamma}{n}$  \ for all $1 \leq j \leq n+1$ and $\gamma = \displaystyle{ \sum_{j=1}^{n+1} } \ \frac{1}{p_{j}}$.
And Lemma 3.2 in \cite{Chen} gives an endpoint case of the multilinear inequality (2.3). That is,
for any nonnegative measurable functions $ f_{j} \in  L^{p_{j}}(\mathbb{R}^{n}) $
$$ \displaystyle{\prod_{j=1}^{n}} \|f_{j}\|_{L^{n, \infty} (\mathbb{R}^{n})} \|f_{n+1}\|_{L^{\infty}}
\leq C_{n} \displaystyle{\sup_{y_{j}}} \  \displaystyle{\prod_{j=1}^{n+1}} \ f_{j}(y_{j}) \  \det(y_{1}, \dots, y_{n+1}).   \eqno (2.4)$$
It is not hard to see (2.4) implies for any $y \in \mathbb{R}^{n}$
$$ \displaystyle{\prod_{j=1}^{n}} \|f_{j}\|_{L^{n, \infty} (\mathbb{R}^{n})}
\leq C_{n} \displaystyle{\sup_{y_{j}}} \  \displaystyle{\prod_{j=1}^{n}} \ f_{j}(y_{j}) \  \det(y, y_{1}, \dots, y_{n}),   \eqno (2.5)$$
which also concludes  (2.2) by specialising to $f_{j}=\chi_{E_{j}}$.
For the proof of (2.3)- (2.5) and more general multilinear cases, we refer to \cite{Chen}.

\bigskip

Before studying the sharp versions of inequalities (2.2),
we recall some useful tools in \cite{Chen}  which were already stated in the introduction.

\begin{lemma} \cite{Chen}  
Let $E_{j}$ be  measurable sets in $\mathbb{R}$ and $a_{j}\in \mathbb{R}$, $j=1, \dots, l$. Then
$$\sup\limits_{x_{j}\in E_{j}^{\ast}}  |\sum_{j=1}^{l} a_{j}x_{j}|  \leq \sup\limits_{x_{j}\in E_{j}} |\sum_{j=1}^{l} a_{j}x_{j}|.  \eqno (2.6)$$

\end{lemma}

\begin{proof}
From the Brunn-Minkowski inequality
$$|E+F| \geq |E| + |F|$$
where $E, F \subset \mathbb{R}$,  it follows that 
$$|E_1+ \dots + E_l| \geq |E_1| + \dots + |E_l|.$$
Because $E_{j}^{\ast}= (-|E_j|/2,  |E_j|/2)$, $1 \leq j \leq l$,
then 
$$E_{1}^{\ast}+ \dots + E_{l}^{\ast} = (-\sum\limits_{j=1}^{l} \frac{|E_j|}{2},  \sum\limits_{j=1}^{l} \frac{|A_j|}{2}).$$
Thus we have
$$|(E_{1}+ \dots + E_{l})^{\ast}| = |E_{1}+ \dots + E_{l}| \geq  |E_1| + \dots + |E_l| =| E_{1}^{\ast}+ \dots + E_{l}^{\ast}|,$$
which implies
$$(E_{1}+ \dots + E_{l})^{\ast} \supset E_{1}^{\ast}+ \dots + E_{l}^{\ast}.  \eqno (2.7)$$
Clearly, for any non-zero $a \in \mathbb{R}$ and  any measurable subset $E$ in $ \mathbb{R}$
$$(aE)^{\ast}=a E^{*}.   \eqno (2.8)$$
Combining with (2.7)-(2.8) we have
$$(a_1 E_{1}+ \dots + a_l E_{l})^{\ast} \supset a_1 E_{1}^{\ast}+ \dots + a_l E_{l}^{\ast}.  \eqno (2.9)$$
Apply  (1.1) and (2.9), 
$$\sup\limits_{x_{j}\in E_{j}}  |\sum_{j=1}^{l} a_{j}x_{j}|  
= \sup\limits_{\bar{x} \in \sum\limits_{j=1}^{l} a_j E_{j}}  |\bar{x}|
\geq \sup\limits_{\bar{x} \in (\sum\limits_{j=1}^{l} a_j E_{j})^{\ast}} |\bar{x}|
\geq \sup\limits_{\bar{x} \in \sum\limits_{j=1}^{l}  a_j E_{j}^{\ast}}  |\bar{x}|.$$
Besides,
$$\sup\limits_{\bar{x} \in \sum\limits_{j=1}^{l}  a_j E_{j}^{\ast}}  |\bar{x}|
= \sup\limits_{x_j \in a_{j}E_{j}^{*} } | \sum_{j=1}^{l}  x_{j}|  
=\sup\limits_{x_j \in E_{j}^{*} }  | \sum_{j=1}^{l}  a_{j}x_{j}|.$$ 
Therefore,
$$\sup\limits_{x_{j}\in E_{j}}  |\sum_{j=1}^{l} a_{j}x_{j}|   \geq \sup\limits_{x \in E_{j}^{*} }  | \sum_{j=1}^{l}  a_{j}x_{j}|.$$

\end{proof}

\medskip

It follows from Lemma 2.2 we have  inequalities (2.10)-(2.12).
Let $E_1, \dots, E_l$ be measurable sets  in $\mathbb{R}^{n}$.
 Let $l \geq n$ and  let  $ A=\{a_{ik}\}$  be an $l \times n$ real matrix. 
Then for each $1 \leq t \leq n$, 
$$\sup\limits_{\substack{y_{j} \in \mathcal{S}_{e_{t}}(E_j) \\  j=1, \dots, l}}  \det(0, \sum\limits_{i=1}^{l} a_{i1}y_{i}, \dots, \sum\limits_{i=1}^{l} a_{in}y_{i})
\leq 
\sup\limits_{\substack{y_{j} \in E_j \\  j=1, \dots, l}}  \det(0, \sum\limits_{i=1}^{l} a_{i1}y_{i}, \dots, \sum\limits_{i=1}^{l} a_{in}y_{i}), \eqno (2.10)$$
where $\{e_{1}, \dots, e_{n}\}$  is   the standard basis for  $\mathbb{R}^{n}$.

Let $l=n$ and
\begin{equation*}
a_{ik}=
\begin{cases}
1 &\mbox{if $i=k$} \\
0 &\mbox{otherwise,}
\end{cases}
\end{equation*}
so (2.10) gives
$$\sup\limits_{\substack{y_{j} \in \mathcal{S}_{e_{t}}(E_j) \\  j=1, \dots, n}} \det(0, y_{1}, \dots, y_{n})
\leq  \sup\limits_{\substack{y_{j} \in E_j \\  j=1, \dots, n}} \det(0, y_{1}, \dots, y_{n}).  \eqno (2.11)$$
If we set $l=n+1$ and
\begin{equation*}
a_{ik}=
\begin{cases}
1 &\mbox{if $i=k$} \\
-1 &\mbox{if $i=n+1$} \\
0 &\mbox{otherwise,}
\end{cases}
\end{equation*}
thus 
$$\sup\limits_{\substack{y_{j} \in \mathcal{S}_{e_{t}}(E_j) \\  j=1, \dots, n+1}} \det(y_{1}, \dots, y_{n+1})
\leq  \sup\limits_{\substack{y_{j} \in E_j \\  j=1, \dots, n+1}} \det(y_{1}, \dots, y_{n+1}).     \eqno (2.12)$$

\begin{proof}
For simplicity, we just see (2.10) holds for $e_{1}$. 
Define the projection $\pi$: $\mathbb{R}^{n} \to \mathbb{R}^{n-1}$ by 
$$\pi(x)=(x_2, \dots, x_n), \ \forall \ x=(x_1, \dots, x_n) \in \mathbb{R}^{n}.$$
For any $x\in  \mathbb{R}^{n}$, write $x= (x_1, x^{\prime})$ where $x^{\prime} \in \mathbb{R}^{n-1}$.
For $y_j \in E_j$,
$$\det(0, y_{1}, \dots, y_{n})=
| \det \left(\begin{array}{cccc}
   y_{11} &  y_{21}   &  \dots  &  y_{n1} \\
   \vdots &  \vdots    &   \   &   \vdots        \\
   y_{1n} &  y_{2n} &  \dots  &  y_{nn} 
 \end{array}
\right) |
= |y_{11} A_1 +y_{21} A_2 + \dots y_{n1} A_{n}|,$$
where $A_{j}$ depend only on $\{y_{1}^{\prime}, \dots, y_{n}^{\prime}\}$.
Hence,
$\det(0, \sum\limits_{i=1}^{l} a_{i1}y_{i}, \dots, \sum\limits_{i=1}^{l} a_{in}y_{i})$ is the linear combination of $y_{11}, \dots, y_{l1}$.
That is,
$$\det(0, \sum\limits_{i=1}^{l} a_{i1}y_{i}, \dots, \sum\limits_{i=1}^{l} a_{in}y_{i})=
 |y_{11} B_1 +y_{21} B_2 + \dots y_{l1} B_{l}|,$$
where $B_{j}$ depend only on $\{y_{1}^{\prime}, \dots, y_{l}^{\prime}\}$.
For each $j$, fix $y_{j}^{\prime}:=(y_{j2}, \dots, y_{jn}) \in \pi(E_j)$, $1 \leq j \leq l$.
Let 
$$E_{j}(y_{j}^{\prime}) =\{y_{j1}\in \mathbb{R}: (y_{j1}, y_{j}^{\prime}) \in E_{j}\}. $$
It follows from Lemma 2.2 that
$$\sup\limits_{y_{j1} \in E_{j}(y_{j}^{\prime})^{\ast}} |\sum_{j=1}^{l} B_{j}y_{j1}|  \leq  \sup\limits_{y_{j1} \in E_{j}(y_{j}^{\prime})} |\sum_{j=1}^{l} B_{j}y_{j1}|. \eqno (2.13)$$
Since
$$\mathcal{S}_{e_{1}}(E_j)= \bigcup\limits_{y_{j}^{\prime} \in \pi (E_j)} \{(y_{j1}, y_{j}^{\prime}):  y_{j1} \in E_{j}(y_{j}^{\prime})^{\ast}\},  $$
together with (2.13)  gives
$$\sup\limits_{\substack{y_{j} \in \mathcal{S}_{e_{1}}(E_j) \\  j=1, \dots, l}}  \det(0, \sum\limits_{i=1}^{l} a_{i1}y_{i}, \dots, \sum\limits_{i=1}^{l} a_{in}y_{i})
\leq 
\sup\limits_{\substack{y_{j} \in E_j \\  j=1, \dots, l}}  \det(0, \sum\limits_{i=1}^{l} a_{i1}y_{i}, \dots, \sum\limits_{i=1}^{l} a_{in}y_{i}).$$

\end{proof}

\medskip

More generally, togehter with the rotation invariance we have the following rearrangement theorem.

\begin{theorem}
Let  $ A=\{a_{ik}\}$  be an $l \times n$ real matrix with $l \geq n$.
Let $u$ be a unit vector in $\mathbb{R}^{n}$. Then
for any measurable sets $E_j \subset \mathbb{R}^{n}$, $1 \leq j \leq l$,
$$\sup\limits_{\substack{y_{j} \in \mathcal{S}_{u}(E_j) \\  j=1, \dots, l}}  \det(0, \sum\limits_{i=1}^{l} a_{i1}y_{i}, \dots, \sum\limits_{i=1}^{l} a_{in}y_{i})
\leq 
\sup\limits_{\substack{y_{j} \in E_j \\  j=1, \dots, l}}  \det(0, \sum\limits_{i=1}^{l} a_{i1}y_{i}, \dots, \sum\limits_{i=1}^{l} a_{in}y_{i}).$$

\end{theorem}

\begin{proof}
Suppose $u=\rho e_{t}$, where $\rho$ is a rotation around the origin in $\mathbb{R}^{n}$. \\
By definition,
\begin{align*}
\mathcal{S}_{\rho e_{t}}(E)
&=\{m \rho e_{t}+y: E\cap [\mathbb{R}(\rho e_{t})+y ] \neq \phi, |m| \leq \frac{|E\cap [\mathbb{R}(\rho e_{t})+y] |}{2}\} \\
&= \{\rho(me_{t}+\rho^{-1}y): \rho^{-1}(E) \cap (\mathbb{R}e_{t}+\rho^{-1}y)\neq \phi, |m| \leq \frac{ |\rho[\rho^{-1}(E)\cap(\mathbb{R}e_{t}+\rho^{-1}y)]|}{2}\} \\
&= \{\rho(me_{t}+\rho^{-1}y): \rho^{-1}(E) \cap (\mathbb{R}e_{t}+\rho^{-1}y)\neq \phi, |m| \leq \frac{ |\rho^{-1}(E)\cap(\mathbb{R}e_{t}+\rho^{-1}y)|}{2}\}.
\end{align*}
Note that
$$\mathcal{S}_{e_{t}}(\rho^{-1}(E))= \{ m e_{t}+\rho^{-1}y: \rho^{-1}(E) \cap (\mathbb{R}e_{t}+\rho^{-1}y)\neq \phi,
|m| \leq \frac{| \rho^{-1}(E)  \cap (\mathbb{R}e_{t}+\rho^{-1}y)|}{2} \}.$$
Hence  we obtain
$$\mathcal{S}_{\rho e_{t}}(E)=\rho \circ \mathcal{S}_{e_{t}}(\rho^{-1}(E)).   \eqno (2.14)$$
By the invariance  under rotation $\rho$ 
\begin{align*}
\sup\limits_{\substack{y_{j} \in \mathcal{S}_{u}(E_j) \\  j=1, \dots, l}} 
\det(0, \sum\limits_{i=1}^{l} a_{i1}y_{i}, \dots, \sum\limits_{i=1}^{l} a_{in}y_{i})
&= \sup\limits_{\substack{y_{j} \in\rho \circ \mathcal{S}_{e_{t}}(\rho^{-1}(E_j)) \\  j=1, \dots, l}} 
\det(0, \sum\limits_{i=1}^{l} a_{i1}y_{i}, \dots, \sum\limits_{i=1}^{l} a_{in}y_{i}) \\
&=\sup\limits_{\substack{y_{j} \in \mathcal{S}_{e_{t}}(\rho^{-1}(E_j)) \\  j=1, \dots, l}} \det(0, \sum\limits_{i=1}^{l} a_{i1}y_{i}, \dots, \sum\limits_{i=1}^{l} a_{in}y_{i}).
\end{align*}
Applying (2.10) gives
\begin{align*}
\sup\limits_{\substack{y_{j} \in \mathcal{S}_{e_{t}}(\rho^{-1}(E_j)) \\  j=1, \dots, l}} \det(0, \sum\limits_{i=1}^{l} a_{i1}y_{i}, \dots, \sum\limits_{i=1}^{l} a_{in}y_{i}) 
&\leq \sup\limits_{\substack{y_{j} \in \rho^{-1}(E_j) \\  j=1, \dots, l}}  \det(0, \sum\limits_{i=1}^{l} a_{i1}y_{i}, \dots, \sum\limits_{i=1}^{l} a_{in}y_{i}) \\
&=  \sup\limits_{\substack{y_{j} \in E_j \\  j=1, \dots, l}} \det(0, \sum\limits_{i=1}^{l} a_{i1}y_{i}, \dots, \sum\limits_{i=1}^{l} a_{in}y_{i}).
\end{align*}
Therefore, we conclude
$$\sup\limits_{\substack{y_{j} \in \mathcal{S}_{u}(E_j) \\  j=1, \dots, l}} 
\det(0, \sum\limits_{i=1}^{l} a_{i1}y_{i}, \dots, \sum\limits_{i=1}^{l} a_{in}y_{i})
=  \sup\limits_{\substack{y_{j} \in E_j \\  j=1, \dots, l}} \det(0, \sum\limits_{i=1}^{l} a_{i1}y_{i}, \dots, \sum\limits_{i=1}^{l} a_{in}y_{i}).$$

\end{proof}

\medskip

Now we can decide the sharp versions of the determinant inqualities in this section.
It is known that, given a compact convex set $K \subset \mathbb{R}^{n}$, there exists a sequence of
iterated Steiner symmetrisations of $K$ that converges in the Hausdorff metric to a ball of the same volume. 
For example, given a basis of unit directions $u_1, \dots, u_n$ for $\mathbb{R}^{n}$ having mutually irrational multiple of $\pi$ radian differences, then 
the sequence $\mathcal{S}_{u_{n}} \dots \mathcal{S}_{u_{2}} \mathcal{S}_{u_{1}}(K)$
iterated infinitely many times to $K$ will converge to a ball of the same volume as $K$.
For the convergence of  Steiner symmetrisation, refer to \cite{Bianchi}, \cite{Bonnesen},
\cite{Eggleston}, \cite{Klain}, \cite{Webster},  etc.

\medskip

One can easily verify that the  suprema function  on the right side of  inequalities (2.10) are continuous 
under the Hausdorff metric, and they do not change if we replace  each $E_j$ by $\overline{\mathrm{co}}(E_j)$.
Therefore, applying the convergence of  Steiner symmetrisation together with  Theorem 2.3 
we have shown the following lemma.

\begin{lemma}

Let $l \geq n$ and  let  $ A=\{a_{ik}\}$  be an $l \times n$ real matrix. 
Then for any measurable sets  $E_{j} \subset \mathbb{R}^{n}$, $1 \leq j \leq l$, 
$$\sup\limits_{y_{1} \in E_{1}^{*}, \dots, y_{l} \in E_{l}^{*} }  \det(0, \sum\limits_{i=1}^{l} a_{i1}y_{i}, \dots, \sum\limits_{i=1}^{l} a_{in}y_{i})
\leq 
\sup\limits_{y_{1} \in E_{1}, \dots, y_{l} \in E_{l} }  \det(0, \sum\limits_{i=1}^{l} a_{i1}y_{i}, \dots, \sum\limits_{i=1}^{l} a_{in}y_{i}).   $$

\end{lemma}

\medskip

Obviously, it follows from Lemma 2.4 that
$$\sup\limits_{y_{1} \in E_{1}^{*}, \dots, y_{n} \in E_{n}^{*} }  \det(0, y_{1}, \dots, y_{n})
\leq 
\sup\limits_{y_{1} \in E_{1}, \dots, y_{n} \in E_{n} }  \det(0, y_{1}, \dots, y_{n}),  \eqno (2.15) $$
and 
$$\sup\limits_{y_{1} \in E_{1}^{*}, \dots, y_{n+1} \in E_{n+1}^{*} }  \det(y_{1}, \dots, y_{n+1})
\leq
\sup\limits_{y_{1} \in E_{1}, \dots, y_{n+1} \in E_{n+1} }  \det(y_{1}, \dots, y_{n+1})   \eqno (2.16) $$
hold for any measurable sets $E_{j} \subset \mathbb{R}^{n}$, $1 \leq j \leq n+1$.

\bigskip

From Lemma 2.4  we obtain the  multilinear functional  rearrangement inequalities.

\begin{theorem}
Let $f_{j}$ be  nonnegative measurable functions vanishing at infinity  on $\mathbb{R}^{n}$. 
Let $A=\{a_{ij}\} \in \mathrm{GL}_{n}(\mathbb{R})$,
then
$$\displaystyle{\sup_{y_{j}}} \prod_{j=1}^{n} f_{j}^{\ast}(\sum\limits_{i=1}^{n}a_{ij}y_{i})  \det(0, y_{1}, \dots, y_{n})
\leq   \displaystyle{\sup_{y_{j}}}  \prod_{j=1}^{n} f_{j}(\sum\limits_{i=1}^{n}a_{ij}y_{i})  \det(0, y_{1}, \dots, y_{n}).   \eqno (2.17)$$
Let $A=\{a_{ij}\} \in \mathrm{GL}_{(n+1)}(\mathbb{R})$, then
$$\displaystyle{\sup_{y_{j}}} \prod_{j=1}^{n+1} f_{j}^{\ast}(\sum\limits_{i=1}^{n+1}a_{ij}y_{i})  \det(y_{1}, \dots, y_{n+1})
\leq   \displaystyle{\sup_{y_{j}}}  \prod_{j=1}^{n+1} f_{j}(\sum\limits_{i=1}^{n+1}a_{ij}y_{i})  \det(y_{1}, \dots, y_{n+1}),   \eqno (2.18)$$
where the $\sup$  is the essential supremum.

\end{theorem}

\begin{proof}
Let $\tilde{y}_j =\sum\limits_{i=1}^{n}a_{ij}y_{i}$, $1 \leq j \leq n$,  so
$$\det(0, y_{1}, \dots, y_{n})= \det(0, \tilde{y}_{1}, \dots, \tilde{y}_{n})  |\det(A)|^{-1}. $$
Then for (2.17) it suffices to prove 
$$\displaystyle{\sup_{\tilde{y}_{j}}} \prod_{j=1}^{n} f_{j}^{\ast}(\tilde{y}_{j})  \det(0, \tilde{y}_{1}, \dots, \tilde{y}_{n})
\leq   \displaystyle{\sup_{\tilde{y}_{j}}}  \prod_{j=1}^{n} f_{j}(\tilde{y}_{j})  \det(0, \tilde{y}_{1}, \dots, \tilde{y}_{n}). \eqno (2.19)$$
Similarly, for (2.18) 
denote  $\tilde{y}_j =\sum\limits_{i=1}^{n+1}a_{ij}y_{i}$, $1 \leq j \leq n+1$.
Since 
$$\left(\begin{array}{ccc}
   y_1  &      \dots  &  y_{n+1}  
 \end{array}
\right) =
\left(\begin{array}{ccc}
  \tilde{y}_1 &      \dots  &  \tilde{y}_{n+1}  
 \end{array}
\right)  A^{-1},$$
$\det(y_{1}, \dots, y_{n+1})$ can be written as the form
$$\det(0, \sum\limits_{i=1}^{n+1} c_{i1} \tilde{y}_{i}, \sum\limits_{i=1}^{n+1} c_{i2} \tilde{y}_{i},   \dots,  \sum\limits_{i=1}^{n+1} c_{in} \tilde{y}_{i}).$$
Specifically, suppose  $A^{-1}=\{b_{ij}\}_{n+1}$, then
by calculation we have $c_{ik}=b_{ik}-b_{i(n+1)}$ with  $1 \leq k \leq n$, $1 \leq i \leq n+1$.
Hence (2.18) becomes 
$$\begin{array}{ll}
& \ \ \ \ \ \displaystyle{\sup_{\tilde{y}_{j}}} \prod_{j=1}^{n+1} f_{j}^{\ast}(\tilde{y}_{j}) 
\det(0, \sum\limits_{i=1}^{n+1} c_{i1} \tilde{y}_{i}, \sum\limits_{i=1}^{n+1} c_{i2} \tilde{y}_{i},   \dots,  \sum\limits_{i=1}^{n+1} c_{in} \tilde{y}_{i}) \\
& \leq   \displaystyle{\sup_{\tilde{y}_{j}}}  \prod_{j=1}^{n+1} f_{j}(\tilde{y}_{j})  
\det(0, \sum\limits_{i=1}^{n+1} c_{i1} \tilde{y}_{i}, \sum\limits_{i=1}^{n+1} c_{i2} \tilde{y}_{i},   \dots,  \sum\limits_{i=1}^{n+1} c_{in} \tilde{y}_{i}).
\end{array}\eqno (2.20)$$

\medskip

We claim that for any  $l \geq n$, for any $l \times n$ real matrix $B=\{c_{ik}\}$
$$\displaystyle{\sup_{y_{j}}} \prod_{j=1}^{l} f_{j}^{\ast}(y_{j})  \det(0, \sum\limits_{i=1}^{l} c_{i1} y_{i}, \dots,  \sum\limits_{i=1}^{l} c_{in} y_{i})
\leq   \displaystyle{\sup_{y_{j}}}  \prod_{j=1}^{l} f_{j}(y_{j})  \det(0, \sum\limits_{i=1}^{l} c_{i1} y_{i}, \dots,  \sum\limits_{i=1}^{l} c_{in} y_{i})  $$
holds.
Suppose 
$$\displaystyle{\sup_{y_{j}}}  \prod_{j=1}^{l} f_{j}(y_{j})  \det(0, \sum\limits_{i=1}^{l} c_{i1} y_{i}, \dots,  \sum\limits_{i=1}^{l} c_{in} y_{i})=s< \infty.$$
We assume for a contradiction  that
$$\displaystyle{\sup_{y_{j}}}  \prod_{j=1}^{l} f_{j}^{\ast}(y_{j})  \det(0, \sum\limits_{i=1}^{l} c_{i1} y_{i}, \dots,  \sum\limits_{i=1}^{l} c_{in} y_{i})>s.$$
Then there exist  positive $\varepsilon$  and a set $ G \subset \mathbb{R}^{n} \times \dots  \times \mathbb{R}^{n}$ such that $|G|>0$ and
for all $(x_1, \dots, x_l) \in G$
we have
$$ \prod_{j=1}^{l} f_{j}^{\ast}(x_{j})  \det(0, \sum\limits_{i=1}^{l} c_{i1} x_{i}, \dots,  \sum\limits_{i=1}^{l} c_{in} x_{i})> s+\varepsilon,  \eqno (2.21)$$
which gives
$$f_{1}^{\ast}(x_1) >   (s+ \varepsilon)     (\prod_{j=2}^{l} f_{j}^{\ast}(x_j) 
\det(0, \sum\limits_{i=1}^{l} c_{i1} x_{i}, \dots,  \sum\limits_{i=1}^{l} c_{in} x_{i}))^{-1}. \eqno (2.22)$$
Define the set 
$$E_{1}:=\{ y_1:   f_{1}(y_1) > (s+\varepsilon)  (\prod_{j=2}^{l} f_{j}^{\ast}(x_j) \det(0, \sum\limits_{i=1}^{l} c_{i1} x_{i}, \dots,  \sum\limits_{i=1}^{l} c_{in} x_{i}))^{-1} \},$$
so by the property of decreasing rearrangement together with (2.22) 
$$ |E_1|> v_{n} |x_{1}|^{n}.    $$
From the definition of $E_1$ 
$$ f_{2}^{\ast}(x_2) >( s+\frac{\varepsilon}{2}) \  (\displaystyle{\inf_{y_{1} \in E_1 }} f_{1}(y_1)  \prod_{j=3}^{l} f_{j}^{\ast}(x_j)  
\det(0, \sum\limits_{i=1}^{l} c_{i1} x_{i}, \dots,  \sum\limits_{i=1}^{l} c_{in} x_{i}) )^{-1}.$$
We then define 
$$E_2= \{ y_2:  f_{2}(y_2)>( s+\frac{\varepsilon}{2}) (\displaystyle{\inf_{y_{1} \in E_1  }} f_{1}(y_1) \prod_{j=3}^{l} f_{j}^{\ast}(x_j)  
\det(0, \sum\limits_{i=1}^{l} c_{i1} x_{i}, \dots,  \sum\limits_{i=1}^{l} c_{in} x_{i}) \},$$
so
$$ |E_2|> v_{n} |x_{2}|^{n}.    $$
Overall, we can take the similar arguments to define sets $E_t$, $1< t< l$
$$E_t =\{ y_t:   f_{t}(y_t) > 
(s+\frac{\varepsilon}{t})  (\prod_{j=1}^{t-1} \inf\limits_{y_{j} \in E_j }   f_{j}(y_j)    \prod_{j=t+1}^{l} f_{j}^{\ast}(x_j) 
\det(0, \sum\limits_{i=1}^{l} c_{i1} x_{i}, \dots,  \sum\limits_{i=1}^{l} c_{in} x_{i}))^{-1} \},$$
and
$$E_l=\{ y_l:   f_{l}(y_l) > (s+\frac{\varepsilon}{l})  (\prod_{j=1}^{l-1} \inf\limits_{y_{j} \in E_j }   f_{j}(y_j)   
\det(0, \sum\limits_{i=1}^{l} c_{i1} x_{i}, \dots,  \sum\limits_{i=1}^{l} c_{in} x_{i}) )^{-1} \}.$$
It is easily seen that for each $j=1, \dots, l$
$$|E_j|> v_{n} |x_{j}|^{n}, \eqno (2.23)$$
and thus $x_j \in E_{j}^{\ast}$.
It follows from Lemma 2.4 that
$$\sup\limits_{y_{1} \in E_{1}^{*}, \dots, y_{l} \in E_{l}^{*}}     \det(0, \sum\limits_{i=1}^{l} c_{i1} y_{i}, \dots,  \sum\limits_{i=1}^{l} c_{in} y_{i}) \leq 
\sup\limits_{y_{1} \in E_1,  \dots, y_{l} \in E_l}  \det(0, \sum\limits_{i=1}^{l} c_{i1} y_{i}, \dots,  \sum\limits_{i=1}^{l} c_{in} y_{i}).$$
That together with $x_{j} \in E_{j}^{\ast}$, $j=1, \dots, l$,  implies
$$ \det(0, \sum\limits_{i=1}^{l} c_{i1} x_{i}, \dots,  \sum\limits_{i=1}^{l} c_{in} x_{i})  \leq 
\sup\limits_{y_{1} \in E_1, \dots, y_{n} \in E_n}  \det(0, \sum\limits_{i=1}^{l} c_{i1} y_{i}, \dots,  \sum\limits_{i=1}^{l} c_{in} y_{i}).  \eqno (2.24)$$
From the definition of $E_l$ we have for any $y_j \in E_j$, $1 \leq j \leq l$
\begin{align*}
& \ \ \ \  \prod\limits_{j=1}^{l} f_{j}(y_{j})   \det(0, \sum\limits_{i=1}^{l} c_{i1} y_{i}, \dots,  \sum\limits_{i=1}^{l} c_{in} y_{i}) \\
&> (s + \frac{\varepsilon}{l} ) (\det(0, \sum\limits_{i=1}^{l} c_{i1} x_{i}, \dots,  \sum\limits_{i=1}^{l} c_{in} x_{i}) )^{-1} 
\det(0, \sum\limits_{i=1}^{l} c_{i1} y_{i}, \dots,  \sum\limits_{i=1}^{l} c_{in} y_{i}).
\end{align*}
Therefore, together with (2.24) we obtain
\begin{align*}
s &\geq  \sup\limits_{y_{1} \in E_1,  \dots, y_{l} \in E_l} \prod_{j=1}^{l} f_{j}(y_{j}) 
\det(0, \sum\limits_{i=1}^{l} c_{i1} y_{i}, \dots,  \sum\limits_{i=1}^{l} c_{in} y_{i}) \\
&>(s + \frac{\varepsilon}{l} ) (\det(0, \sum\limits_{i=1}^{l} c_{i1} x_{i}, \dots,  \sum\limits_{i=1}^{l} c_{in} x_{i}) )^{-1}  
\sup\limits_{y_{1} \in E_1,  \dots, y_{l} \in E_l} 
\det(0, \sum\limits_{i=1}^{l} c_{i1} y_{i}, \dots,  \sum\limits_{i=1}^{l} c_{in} y_{i}) \\
&>s,
\end{align*}
which gives a contradiction.
That completes the proof of claim.
Therefore, (2.19)-(2.20) hold.

\end{proof}

\medskip

\begin{remark}
We use a counterexample to show that Theorem 2.5  is false if $\det(A)=0$.
Let $f_1=\chi_{A}$, $f_2= \chi_{B}$ where $A, B$ are disjoint  measurable sets in $\mathbb{R}^{2}$ with non-zero measure.
Obviously,
$$\sup\limits_{y_1, y_2 \in \mathbb{R}^{2}} f_1 (y_1 +y_2) f_2 (y_1 +y_2) \det(0, y_{1},  y_{2})=0,$$
while 
$$\sup\limits_{y_1, y_2} f_{1}^{\ast} (y_1 +y_2) f_{2}^{\ast} (y_1 +y_2) \det(0, y_{1},  y_{2}) \neq 0.$$
Likewise,
for the same sets $A, B$ above, let $f_1=\chi_{A}$, $f_2=f_3= \chi_{B}$.
Then 
$$\sup\limits_{y_1, y_2, y_3 \in \mathbb{R}^{2}} f_1 (y_1 +y_2+y_3) f_2 (y_1 +y_2+y_3)f_ 3(y_3) \det(y_{1},  y_{2}, y_{3})=0,$$
while
$$\sup\limits_{y_1, y_2, y_3 \in \mathbb{R}^{2}} f_{1}^{\ast} (y_1 +y_2+y_3) f_{2}^{\ast} (y_1 +y_2+y_3) f_{3}^{\ast}(y_3) \det(y_{1},  y_{2}, y_{3}) \neq 0.$$

\end{remark}

\bigskip

Let $A=I$. From Theorem 2.5 it is straightforward to see that
$$\displaystyle{\sup_{y_{j}}} \prod_{j=1}^{n} f_{j}^{\ast}(y_{j})  \det(0, y_{1}, \dots, y_{n})
\leq   \displaystyle{\sup_{y_{j}}}  \prod_{j=1}^{n} f_{j}(y_{j})  \det(0, y_{1}, \dots, y_{n}),  \eqno (2.25)$$
$$\displaystyle{\sup_{y_{j}}} \prod_{j=1}^{n+1} f_{j}^{\ast}(y_{j})  \det(y_{1}, \dots, y_{n+1})
\leq   \displaystyle{\sup_{y_{j}}}  \prod_{j=1}^{n+1} f_{j}(y_{j})  \det(y_{1}, \dots, y_{n+1}).  \eqno (2.26)$$

Let $f_{j}= \chi_{E_{j}}$, and $E_j$ be measurable sets in $\mathbb{R}^{n}$.
Applying  (2.25)-(2.26)  we obtain  the following  two sharp ``multilinear" determinant inequalties suggested by the multilinear perspective of (2.2):
$$ \displaystyle{\prod_{j=1}^{n}} \  |E_{j}|^{\frac{1}{n}} \leq A_n
\sup\limits_{y_{1} \in E_{1}, \dots, y_{n} \in E_{n} }  \det(0, y_{1}, \dots, y_{n}),  \eqno (2.27) $$
and
$$ \displaystyle{\prod_{j=1}^{n+1}} \  |E_{j}|^{\frac{1}{n+1}} \leq B_n
\sup\limits_{y_{1} \in E_{1}, \dots, y_{n+1} \in E_{n+1} }  \det(y_{1}, \dots, y_{n+1}).   \eqno (2.28) $$
Moreover, they are both  extremised by balls centred at $0$.
It follows from (2.25)-(2.26) that we also obtain the optimisers for (1.7) and (1.8) which is the special case when $E_{j}=E$.

\medskip

It should be pointed out that  (2.25)-(2.26) improves multilinear rearrangement inequalities (2.29), (2.30) given in \cite{Chen}.
For each $1 \leq i \leq n$
$$\displaystyle{\sup_{y_{j}}} \prod_{j=1}^{n} f_{j}^{\ast i}(y_{j})  \det(0, y_{1}, \dots, y_{n})
\leq   \displaystyle{\sup_{y_{j}}}  \prod_{j=1}^{n} f_{j}(y_{j})  \det(0, y_{1}, \dots, y_{n}),  \eqno (2.29)$$
and
$$\displaystyle{\sup_{y_{j}}} \prod_{j=1}^{n+1} f_{j}^{\ast i}(y_{j})  \det(y_{1}, \dots, y_{n+1})
\leq   \displaystyle{\sup_{y_{j}}}  \prod_{j=1}^{n+1} f_{j}(y_{j})  \det(y_{1}, \dots, y_{n+1}),  \eqno (2.30)$$
where $f_{j}^{\ast i}$  is the Steiner symmetrisation of $f_j$  with respect to the $i$-th coordinate.

\medskip

Finally we give the best constant of inequality (2.1) mainly  applying the Brascamp-Lieb-Luttinger rearrangement inequality. 
In 1974, Brascamp, Lieb and Luttinger \cite{BLL} proved the following inequality  (2.31)
which is a generalisation of Riesz's rearrangement inequality \cite{Riesz}.  

Let $f_{j}$ be nonnegative measurable functions on $\mathbb{R}^{n}$ that vanish at infinity, $j=1, \dots, m$.
Let $k \leq m$  and let  $B=\{b_{ij}\}$ be a $k\times m$  matrix with $1 \leq i \leq k$,  $1 \leq j \leq m$. 
Define
$$I(f_{1}, \dots, f_{m}):= \int_{(\mathbb{R}^{n})^{k}} \prod\limits_{j=1}^{m} f_{j}(\sum\limits_{i=1}^{k} b_{ij}x_{i}) dx_{1} \dots dx_{k}.$$
Then
$$I(f_{1}, \dots, f_{m}) \leq I(f_{1}^{\ast}, \dots, f_{m}^{\ast}).    \eqno(2.31)$$

\medskip

\begin{theorem}
Let $f_{j}$ be  nonnegative measurable functions vanishing at infinity  on $\mathbb{R}^{n}$,  
Define 
$$J(f_1, \dots, f_{n+1})=\int_{(\mathbb{R}^{n})^{n}}  \prod_{j=1}^{n} f_{j}(y_{j})  f_{n+1} (\det(0, y_{1}, \dots, y_{n}))  dy_{1} \dots dy_{n}$$
and 
$$G(f_1, \dots, f_{n+2})= \int_{(\mathbb{R}^{n})^{n+1}}  \prod_{j=1}^{n+1} f_{j}(y_{j})  f_{n+2} (\det(y_{1}, \dots, y_{n+1}))  dy_{1} \dots dy_{n+1}$$
Then 
$$J(f_1, \dots, f_{n+1}) \leq J(f_{1}^{\ast}, \dots, f_{n+1}^{\ast}), \eqno (2.32)$$
and
$$G(f_1, \dots, f_{n+2}) \leq G(f_{1}^{\ast}, \dots, f_{n+2}^{\ast}).  \eqno (2.33)$$

\end{theorem}

\begin{proof}
By the layer cake representation, it suffices to show that for any $E_{j}$ of finite volume in $\mathbb{R}^{n}$,
$1 \leq j \leq n+2$,
$$J(E_{1}, \dots, E_{n+1}) \leq J(E_{1}^{\ast}, \dots, E_{n+1}^{\ast}),   \ \  
G(E_{1}, \dots, E_{n+2}) \leq G(E_{1}^{\ast}, \dots, E_{n+2}^{\ast}).  $$
For any measurable $F_{j} \subset \mathbb{R}$, $1 \leq j \leq n+1$, 
Brascamp-Lieb-Luttinger rearrangement inequality implies that
\begin{align*}
& \ \ \ \int_{(\mathbb{R}^{n})^{n}} \prod\limits_{j=1}^{n} \chi_{F_j}(x_j) \chi_{F_{n+1}} (\sum\limits_{j=1}^{n} a_{j}x_{j} ) dx_{1} \dots  dx_{n} \\
&\leq \int_{(\mathbb{R}^{n})^{n}} \prod\limits_{j=1}^{n} \chi_{F_{j}^{\ast}}(x_j) \chi_{F_{n+1}^{*}} (\sum\limits_{j=1}^{n} a_{j}x_{j} ) dx_{1} \dots  dx_{n}.
\end{align*}
As before, since $\det(0, y_{1}, \dots, y_{n})$ is the linear combination of $y_{11}, \dots, y_{n1}$,
similar to the proof of (2.10) 
we have
$$J(E_{1}, \dots, E_{n+1}) \leq J(\mathcal{S}_{e_1}(E_1), \dots, \mathcal{S}_{e_1}(E_{n+1}) ). \eqno(2.34)$$
Note that $J(E_{1}, \dots, E_{n+1})$ is invariant under  $O(n)$.
By the property of
$$\mathcal{S}_{\rho e_{i}}(E)=\rho \circ \mathcal{S}_{e_{i}}(\rho^{-1}(E)),$$
we obtain for any $u \in \mathbb{S}^{n-1}$ that is a unit vector in $\mathbb{R}^{n}$,
$$J(E_{1}, \dots, E_{n+1}) \leq J(\mathcal{S}_{u}(E_1), \dots, \mathcal{S}_{u}(E_{n+1}) ).  \eqno(2.35)$$
Likewise, since $\det( y_{1}, \dots, y_{n+1})$ can be seen as the linear combination of $y_{11}, \dots, y_{(n+1)1}$,
and  the Brascamp-Lieb-Luttinger rearrangement inequality 
\begin{align*}
& \ \ \ \int_{(\mathbb{R}^{n})^{n}} \prod\limits_{j=1}^{n+1} \chi_{F_j}(x_j) \chi_{F_{n+2}} (\sum\limits_{j=1}^{n+1} a_{j}x_{j} ) dx_{1} \dots  dx_{n+1} \\
&\leq \int_{(\mathbb{R}^{n})^{n}} \prod\limits_{j=1}^{n+1} \chi_{F_{j}^{\ast}}(x_j) \chi_{F_{n+2}^{*}} (\sum\limits_{j=1}^{n+1} a_{j}x_{j} ) dx_{1} \dots  dx_{n+1},
\end{align*}
we also have
$$G(E_{1}, \dots, E_{n+2}) \leq G(\mathcal{S}_{e_1}(E_1), \dots, \mathcal{S}_{e_1}(E_{n+2})).  \eqno(2.36)$$
Hence by (2.14) together with the invariance of $G(E_{1}, \dots, E_{n+2})$ 
$$G(E_{1}, \dots, E_{n+2}) \leq G(\mathcal{S}_{u}(E_1), \dots, \mathcal{S}_{u}(E_{n+2})).  \eqno(2.37)$$

\medskip

Let $H$ be the semigroup of all finite products of $\mathcal{S}_{u}$'s.
Brascamp, Lieb and Luttinger \cite{BLL} proved for any bounded measurable $E \subset \mathbb{R}^{n}$,
there exists  $\{h_m\}_{m=0}^{\infty} \subset G$ such that $E_{m}:=h_m (E)$ converges to $E^{\ast}$ in symmetric difference.
That is, 
$$\lim\limits_{m \to \infty } |E_m \triangle E^{\ast}| =0,  \eqno(2.38)$$
where $\triangle$ denotes the symmetric difference of two sets.
Here we sketch the sequence of sets $\{E_m\}$.
Let $E_{0}=h_0 E=E$. Given $E_{m}$, choose unit vector $u_{1}$ such that
$$|\mathcal{S}_{u_1} (E_m)  \triangle E^{\ast}|  < \inf\limits_{u \in \mathbb{S}^{n-1}} |\mathcal{S}_{u} (E_m)  \triangle E^{\ast}| + \frac{1}{m}.$$
Hence we select  $u_2$, \dots, $u_n \in \mathbb{S}^{n-1}$  such that $\{u_1, \dots, u_n\}$ becomes  an orthonormal basis in $\mathbb{R}^{n}$,
and then construct 
$$E_{m+1}=h_{m+1} (E)= \mathcal{S}_{u_n} \mathcal{S}_{u_{n-1}} \dots  \mathcal{S}_{u_1}   (E_m).$$ 
The sequence of sets $\{E_m\}$ constructed above  converges to  $E^{\ast}$ in symmetric difference. See \cite{BLL} for the detailed proof.
Therefore,  we apply the convergence of Steiner symmetrisation together with (2.35) and (2.37) to  conclude
$$J(E_1, \dots, E_{n+1}) \leq J(E_{1}^{\ast}, \dots, E_{n+1}^{\ast}),$$
and
$$G(E_1, \dots, E_{n+2}) \leq G(E_{1}^{\ast}, \dots, E_{n+2}^{\ast}).  $$

Lastly, applying the layer cake representation for $f_j$ together with  Fubini's theorem  gives
$$J(f_1, \dots, f_{n+1})
=\int_{0}^{\infty} \dots \int_{0}^{\infty} 
J( \chi_{\{f_1 >t_1\}}, \dots, \chi_{\{f_{n+1} >t_{n+1}\}})dt_{1} \dots dt_{n+1}.$$
Since (2.32)-(2.33) hold for characteristic functions of sets of finite Lebesgue measure, 
for any $t_{j}$, $1\leq j \leq n+1$
$$J( \chi_{\{f_1 >t_1\}}, \dots, \chi_{\{f_{n+1} >t_{n+1}\}})
\leq J( \chi_{\{f_1 >t_1\}}^{\ast}, \dots, \chi_{\{f_{n+1} >t_{n+1}\}}^{\ast}).   \eqno(2.39)$$
Thus
\begin{align*}
J(f_1, \dots, f_{n+1}) &\leq \int_{0}^{\infty} \dots \int_{0}^{\infty} 
 J( \chi_{\{f_1 >t_1\}}^{\ast}, \dots, \chi_{\{f_{n+1} >t_{n+1}\}}^{\ast}) dt_{1} \dots dt_{n+1} \\
&= J(f_{1}^{\ast}, \dots, f_{n+1}^{\ast}).
\end{align*}
Similarly, 
$$G(f_1, \dots, f_{n+2}) \leq  G(f_{1}^{\ast}, \dots, f_{n+2}^{\ast}).$$
This completes Theorem 2.7.

\end{proof}

Let $f_j= \chi_{E_{j}}$, $1 \leq j \leq n$, and $f_{n+1}=\chi_{ (| \cdot | < \delta)}$. Theorem 2.7 gives
\begin{align*}
& \ \ \ \ \  |\{(y_{1}, \dots, y_{n})\in E_{1} \times \cdots \times E_{n}:
 \det( 0, y_{1}, \dots, y_{n})  < \delta \}|  \\
&\leq  | \{(y_{1}, \dots, y_{n})\in E_1^{\ast} \times \cdots \times E_n^{\ast} : \det(0, y_{1}, \dots, y_{n})  < \delta \}|.
\end{align*}
This  implies that  inequality (2.1) is extremised by balls centred at $y$, where $y\in \mathbb{R}^{n}$.

Let $f_{n+2}=| \cdot |^{-1}$, then Theorem 2.7 implies 
\begin{align*}
& \ \ \ \ \ \ \int_{(\mathbb{R}^{n})^{n+1}}  \prod_{j=1}^{n+1} f_{j}(y_{j}) \det(y_{1}, \dots, y_{n+1})^{-1}  dy_{1} \dots dy_{n+1} \\
&\leq 
\int_{(\mathbb{R}^{n})^{n+1}}  \prod_{j=1}^{n+1} f_{j}^{\ast}(y_{j}) \det(y_{1}, \dots, y_{n+1})^{-1}  dy_{1} \dots dy_{n+1}.
\end{align*}

\medskip

\section{ Matrix inequalities}

Now we turn to see the analogues of (1.5) and (1.6) replacing the  Euclidean space $\mathbb{R}^{n}$  by the space of $n \times n$ real matrices.
We remark that the proof of Theorem 3.1 mainly relies on the rearrangement inequality (2.6) and an invariance under the action of $O(n)$
by premultiplication  as described in the introduction.

\begin{theorem}
There exists a finite constant $\mathcal{C}_{n}$ such that  for any  measurable set $E_j \subset \mathfrak{M}^{n \times n}$ of finite measure, $j=1, \dots, n$,
$$ \prod\limits_{j=1}^{n}|E_j|^{\frac{1}{n^{2}}}  
\leq \mathcal{C}_{n} \sup\limits_{\substack{A_{j} \in E_{j}  \\  j=1, \dots, n}}  | \det(A_{1}+ \dots + A_{n} ) |, \eqno(3.1)$$
where $|\cdot |$ denotes the Lebesgue measure on Euclidean space $\mathbb{R}^{n^{2}}$ and
the absolute value on $\mathbb{R}$.

\end{theorem}

\begin{proof}
Suppose  
$$\sup\limits_{\substack{A_{j} \in E_{j}  \\  j=1, \dots, n}}  | \det(A_{1}+ \dots + A_{n} ) |=s < \infty.$$
First we give some definition and notation. 
Let $F \subset  \mathfrak{M}^{n \times m}$,  define 
\begin{align*}
v(F)=\{
\left(\begin{array}{cccc}
   a_{11} &  a_{21}   &  \dots  &  a_{(m-1)1} \\
   \vdots &  \vdots    &   \   &   \vdots        \\
   a_{1n} &  a_{2n} &  \dots  &  a_{(m-1)n} 
 \end{array}
\right): \exists \
\left(\begin{array}{c}
   a_{m1} \\  \vdots  \\ a_{mn}
 \end{array} \right)
 \ \mathrm{such \  that} \
\left(\begin{array}{ccc}
  a_{11} &  \dots & a_{m1} \\
  \vdots &  \   & \vdots \\
  a_{1n} &  \dots & a_{mn}
 \end{array} \right)
 \in F \},
\end{align*}
so $ v(F) \subset \mathfrak{M}^{n \times (m-1)}$.
For any $n$-by-$(m-1)$ matrix 
$$x=\left(\begin{array}{cccc}
   a_{11} &  a_{21}   &  \dots  &  a_{(m-1)1} \\
   \vdots &  \vdots    &   \   &   \vdots        \\
   a_{1n} &  a_{2n} &  \dots  &  a_{(m-1)n} 
 \end{array}
\right) \in v(F),$$ 
we denote
\begin{align*}
F^{x}=\{\left(\begin{array}{c}
   a_{m1} \\  \vdots  \\ a_{mn}
 \end{array} \right): \
 \left(\begin{array}{ccc}
  a_{11} &  \dots & a_{m1} \\
  \vdots &  \   & \vdots \\
  a_{1n} &  \dots & a_{mn}
 \end{array} \right)
 \in F \} \subset \mathfrak{M}^{n \times 1}.
\end{align*}

Let $E \subset \mathfrak{M}^{n \times n}$.
For any rotation around the origin $T$ in $\mathbb{R}^{n}$, consider
$$\Phi_{T} : A \mapsto TA, \  \forall \  A \in E,$$
where $T$ is a $n$-by-$n$ matrix with $\det(T)=1$.
Note that $\Phi_{T}$ does not change $|E|$ and $\sup\limits_{A \in E} |\det(A)|$. This is because
$$\sup\limits_{A \in \Phi_{T}(E)} |\det(A)|= \sup\limits_{A \in E}  |\det(TA)|= \sup\limits_{A \in E} |\det(A)|.  \eqno (3.2)$$
Besides, if we see the matrix 
$A= \left(\begin{array}{ccc}
  a_{11} &  \dots & a_{n1} \\
  \vdots &  \   & \vdots \\
  a_{1n} &  \dots & a_{nn}
 \end{array} \right) \in E$ as a vector 
$$(a_{11}, \dots, a_{1n}, a_{21}, \dots, a_{2n}, \dots, a_{n1}, \dots, a_{nn}) \in \mathbb{R}^{n^{2}},$$
then the matrix  $\Phi_{T}(A)$  becomes
\begin{align*}
 \left(\begin{array}{cccc}
 T &  \  & \  & \     \\
  \  &  T  & \  & \    \\
 \  &   \  & \ddots  & \  \\  
  \  &  \  &  \  &  T
 \end{array} \right)
\left(\begin{array}{c}
   a_{11} \\  \vdots  \\ a_{nn}
 \end{array} \right).
\end{align*}
Thus 
$$|\Phi_{T}(E)|=|T|^{n}|E|=|E|.  \eqno (3.3)$$
From $|E|=\int_{v(E)}|E^{x}|dx$ it follows that
there always exists $\overline{x} \in v(E)$ such that
$$|v(E)| |E^{\overline{x}}| \gtrsim_n  |E|. \eqno (3.4)$$

By John Ellipsoid,  for any compact convex $G\subset \mathbb{R}^{n}$ there exists an ellipsoid  $G^{\prime} \subset G$ such that
$$|G^{\prime}|  \gtrsim_n  |G|.   \eqno (3.5) $$
For the John ellipsoid $G^{\prime}$, we choose  
a rotation $T \in O(n)$  such that $T G^{\prime}$ is an ellipsoid with principal axes parallel to the coordinate axes.
As well known, for every ellipsoid $T G^{\prime}$ with principal  axes parallel to the coordinate axes,
there exists an axis-parallel rectangle $H \subset T G^{\prime} $ such that 
$$|H|  \gtrsim_n |T G^{\prime}|.   \eqno (3.6) $$
Hence if $E^{\overline{x}}$ is convex,  from (3.5)-(3.6) we may assume that there exists $T \in O(n)$ such that $E^{\overline{x}}$
is an axis-parallel rectangle in $\mathbb{R}^{n}$.

\medskip

Take $n=2$.
By (3.4) there exists $x_{10}\in v(E_{1}) \subset  \mathfrak{M}^{2 \times 1}$, $x_{20}\in v(E_{2}) \subset  \mathfrak{M}^{2 \times 1}$ such that 
$$|v(E_{1})| |E_{1}^{x_{10}}| \gtrsim  |E_1|,  \ \  |v(E_{2})| |E_{2}^{x_{20}}| \gtrsim  |E_2|.   \eqno (3.7)$$
Then 
$$\max \{|v(E_{2})| |E_{1}^{x_{10}}|, |v(E_{1})| |E_{2}^{x_{20}}| \} \gtrsim (|E_1||E_2|)^{1/2}.$$
For simplicity, suppose 
$$|v(E_{2})| |E_{1}^{x_{10}}| \gtrsim (|E_1||E_2|)^{1/2}.  \eqno (3.8)$$

To study the suprema, we consider $2$-by-$2$ matrix
$$\overline{A}_{1}:=
\left(
\begin{array}{cc}
(x_{10})_1  &    (x_{10})_2   \\
\end{array}
\right) \in E_{1}$$
with 
\begin{center}
$(x_{10})_1   =x_{10} \in \mathfrak{M}^{n \times 1}$
 and  \ 
$(x_{10})_2 \in E_{1}^{x_{10}}$.
\end{center}
 For any 
$\overline{A}_{2}:=
\left(
\begin{array}{cc}
x_1 &    x_2  \\
\end{array}
\right) \in E_{2}$, 
 for any constructed $\overline{A}_{1}$ above
\begin{align*}
s &\geq | \det (\overline{A}_{1} +  \overline{A}_{2}) | \\
&=|  \det
\left(
\begin{array}{ccc}
x_1+ (x_{10})_1 &  x_2+ (x_{10})_2   \\
\end{array}
\right)|.
\end{align*}
So  fix the first column, we have for any $x_1 \in v(E_2)$, $x_2 \in E_{2}^{x_1}$
$$s \geq   \sup\limits_{(x_{10})_{2} \in E_{1}^{x_{10}}}
|  \det
\left(
\begin{array}{cc}
x_1+ (x_{10})_1 &  x_2+ (x_{10})_2   \\
\end{array}
\right)|. \eqno (3.9)$$
Because fix  all the columns except one,  the $|\det|$ function is convex function of the remaining column.
Thus
$$s \geq   \sup\limits_{(x_{10})_{2} \in \mathrm{co}E_{1}^{x_{10}}}
|  \det
\left(
\begin{array}{cc}
x_1+ (x_{10})_1 &  x_2+ (x_{10})_2   \\
\end{array}
\right)|. \eqno (3.10)$$
 By (3.5) we may assume $\mathrm{co}E_{1}^{x_{10}}$ is an ellipsoid  in $\mathbb{R}^{2}$.
Choose a rotation $T_{0} \in O(2)$  such that  $T_{0} \mathrm{co}E_{1}^{x_{10}}$ is an ellipsoid with principal axes parallel to the coordinate axes.
From (3.6) we  may assume  $T_{0} \mathrm{co}E_{1}^{x_{10}}$  is an axis-parallel rectangle.
Note that (3.10) is invariant under $O(2)$ as discussed in (3.2), so 
$$\begin{array}{ll}
s &\geq  \sup\limits_{(x_{10})_{2} \in \mathrm{co}E_{1}^{x_{10}}}
|  \det
\left(
\begin{array}{cc}
x_1+ (x_{10})_1 &  x_2+ (x_{10})_2   \\
\end{array}
\right)|   \\
&=\sup\limits_{(x_{10})_{2} \in \mathrm{co}E_{1}^{x_{10}}}
|  \det
\left(
\begin{array}{cc}
T_{0} x_1+ T_0 (x_{10})_1 &  T_0 x_2+ T_0 (x_{10})_2   \\
\end{array}
\right)|. 
\end{array}\eqno(3.11)$$
Since $T_{0} \mathrm{co}E_{1}^{x_{10}}$   is an axis-parallel rectangle  in $\mathbb{R}^{2}$,
it can be written as $A_1 \times A_2$, where $A_1, A_2$ are intervals in  $\mathbb{R}$, and then
$$\mathcal{S}(T_{0} \mathrm{co}E_{1}^{x_{10}})=\mathcal{S}(T_{0} \mathrm{co}E_{1}^{x_{10}}+ T_0 x_2 )=A_{1}^{\ast} \times A_{2}^{\ast},  \  \forall \ x_2\in E_{2}^{x_1}.$$
Similar to the proof of  (2.10),
applying (2.6) gives for any $x_1 \in v(E_{2})$
$$s \geq   \sup\limits_{(x_{10})_{2} \in \mathcal{S} (T_0 \mathrm{co}E_{1}^{x_{10}})}
|  \det
\left(
\begin{array}{cc}
T_{0} x_1+ T_0 (x_{10})_1 & (x_{10})_{2}   \\
\end{array}
\right)|. \eqno (3.12)$$
Therefore, by (2.2) we deduce that
$$s \geq C |T_{0} v(E_{2})+ T_0 (x_{10})_1 |^{1/2} |\mathcal{S} (T_0 \mathrm{co}E_{1}^{x_{10}})|^{1/2}
= C |v(E_{2})|^{1/2} |\mathrm{co}E_{1}^{x_{10}}|^{1/2}.$$
This together with (3.8) implies
$$s  \geq C |v(E_{2})|^{1/2} |\mathrm{co}E_{1}^{x_{10}}|^{1/2} \geq C  |v(E_{2})|^{1/2} |E_{1}^{x_{10}}|^{1/2}  \geq C(|E_1||E_2|)^{1/2},$$
which completes (3.1) for $n=2$.

\medskip

Take $n=3$. By (3.4) for each $E_{j}$ there exists $x_{j0} \in v(E_{j}) \subset \mathfrak{M}^{3 \times 2}$ such that
$$|v(E_{j})| |E_{j}^{x_{j0}}| \gtrsim  |E_j|, 1 \leq j \leq 3.   \eqno (3.13)$$
Denote $F_{j}=v(E_{j}) \subset \mathfrak{M}^{3 \times 2}$,
there exists fixed $x_{j1} \in v(F_{j}) \subset \mathfrak{M}^{3 \times 1}$ such that
$$|v(F_{j})| |F_{j}^{x_{j1}}| \gtrsim |F_{j}| =v(E_{j}).  \eqno (3.14)$$
From (3.13)-(3.14), we have $x_{j0} \in v(E_{j}) , x_{j1} \in v(F_{j})$
$$|v(F_{j})|  |F_{j}^{x_{j1}}| |E_{j}^{x_{j0}}| \gtrsim  |E_j|,  1 \leq j \leq 3. \eqno (3.15)$$
It is not hard to see there exists $\{i_1, i_2, i_3\}$ with $i_1 \neq i_2 \neq i_3$  such that
$$(v(F_{i_3})|  |F_{i_2}^{x_{i_2 1}}|  |E_{i_1}^{x_{i_1 0}}| )^{3}
\geq   \prod\limits_{j=1}^{3}   (|v(F_{j})|  |F_{j}^{x_{j1}}| |E_{j}^{x_{j0}}|)
\gtrsim   \prod\limits_{j=1}^{3} |E_j|.  \eqno (3.16)$$
For simplicity, suppose 
$$|v(F_{3})| |F_{2}^{x_{21}}|  |E_{1}^{x_{10}}| \gtrsim (|E_1||E_2||E_3|)^{1/3}.  \eqno (3.17)$$

Now we consider $3$-by-$3$ matrices
$$\overline{A}_{1}:=
\left(
\begin{array}{ccc}
(x_{10})_1  &    (x_{10})_2  &   (x_{10})_3   \\
\end{array}
\right) \in E_{1}$$
with 
$\left(
\begin{array}{cc}
(x_{10})_1   & (x_{10})_2   \\
\end{array}
\right) =x_{10} \in \mathfrak{M}^{3 \times 2}$ and $(x_{10})_3 \in E_{1}^{x_{10}}$;
$$\overline{A}_{2}:=
\left(
\begin{array}{ccc}
(x_{21})_1  &    (x_{21})_2   &   (x_{21})_3  \\
\end{array}
\right) \in E_{2}$$
with the condition 
\begin{center}
$(x_{21})_1   =x_{21} \in \mathfrak{M}^{3 \times 1}$
 and  \ 
$(x_{21})_2 \in F_{2}^{x_{21}}$.
\end{center}
For any 
$\overline{A}_{3}:=
\left(
\begin{array}{ccc}
x_1 &    x_2   &    x_3 \\
\end{array}
\right) \in E_{3}$, 
for any constructed $\overline{A}_{1}, \overline{A}_{2}$ above,
\begin{align*}
s &\geq | \det (\overline{A}_{1} +  \overline{A}_{2}+  \overline{A}_{3}  )   | \\
&=|  \det
\left(
\begin{array}{ccc}
x_1+ (x_{10})_1+  (x_{21})_1  &  x_2+ (x_{10})_2  + (x_{21})_2   &  x_3+ (x_{10})_3 + (x_{21})_3 \\
\end{array}
\right)|.
\end{align*}
So  fix all columns except the $3$rd column, we have
$$s \geq   \sup\limits_{(x_{10})_{3} \in E_{1}^{x_{10}}}
|  \det
\left(
\begin{array}{ccc}
x_1+ (x_{10})_1+  (x_{21})_1  &  x_2+ (x_{10})_2  + (x_{21})_2   &  x_3+ (x_{10})_3 + (x_{21})_3 \\
\end{array}
\right)|.$$
Obviously,
$$s \geq   \sup\limits_{(x_{10})_{3} \in \mathrm{co}E_{1}^{x_{10}}}
|  \det
\left(
\begin{array}{ccc}
x_1+ (x_{10})_1+  (x_{21})_1  &  x_2+ (x_{10})_2  + (x_{21})_2   &  x_3+ (x_{10})_3 + (x_{21})_3 \\
\end{array}
\right)|. $$
As before, by (3.5)  we assume there exists 
$T_{0} \mathrm{co}E_{1}^{x_{10}}$ is an ellipsoid with principal axes parallel to the coordinate axes in $\mathbb{R}^{3}$.
From (3.6) we  may assume  $T_{0} \mathrm{co}E_{1}^{x_{10}}$  is an axis-parallel rectangle.
Because of the  invariance under $O(3)$,
\begin{align*}
s &\geq   
 \sup\limits_{(x_{10})_{3} \in \mathrm{co}E_{1}^{x_{10}}}
|  \det
\left(
\begin{array}{ccc}
x_1+ (x_{10})_1+  (x_{21})_1  &  x_2+ (x_{10})_2  + (x_{21})_2   &  x_3+ (x_{10})_3 + (x_{21})_3 \\
\end{array}
\right)| \\
&=\sup\limits_{(x_{10})_{3} \in \mathrm{co}E_{1}^{x_{10}}}
|  \det
\left(
\begin{array}{ccc}
T_{0}(x_1+ (x_{10})_1+  (x_{21})_1)  & T_{0}( x_2+ (x_{10})_2  + (x_{21})_2 )  &  T_{0} (x_3+ (x_{10})_3 + (x_{21})_3) \\
\end{array}
\right)|. 
\end{align*}
Since $T_{0} \mathrm{co}E_{1}^{x_{10}}$   is an axis-parallel rectangle  in $\mathbb{R}^{3}$,
it can be written as $A_1 \times A_2 \times A_3$, where $A_1, A_2, A_3$ are intervals in  $\mathbb{R}$.
Similar to the proof of  (2.10) together with 
$$\mathcal{S}(T_{0} \mathrm{co}E_{1}^{x_{10}})=\mathcal{S}(T_{0} \mathrm{co}E_{1}^{x_{10}}+h)
=A_{1}^{\ast} \times A_{2}^{\ast} \times A_{3}^{\ast}, \ \forall \ h \in \mathbb{R}^{3},$$
applying (2.6) gives  for any
$\left(
\begin{array}{cc}
x_1 &    x_2    \\
\end{array}
\right) \in v(E_{3})$, 
$$s \geq   \sup\limits_{(x_{10})_{3} \in \mathcal{S} (T_0 \mathrm{co}E_{1}^{x_{10}})}
|  \det
\left(
\begin{array}{ccc}
T_{0}(x_1+ (x_{10})_1+  (x_{21})_1)  & T_{0}( x_2+ (x_{10})_2  + (x_{21})_2 )  &  (x_{10})_{3}  \\
\end{array}
\right)|. $$

Then fix all columns except the $2$nd column, 
$$s \geq   \sup\limits_{(x_{21})_{2} \in F_{2}^{x_{21}}}
|  \det
\left(
\begin{array}{ccc}
T_{0}(x_1+ (x_{10})_1+  (x_{21})_1)  & T_{0}( x_2+ (x_{10})_2  + (x_{21})_2 )  &  (x_{10})_{3}  \\
\end{array}
\right)| $$
holds for any $(x_{10})_{3} \in \mathcal{S} (T_0 \mathrm{co}E_{1}^{x_{10}})$.
Similarly, by the convex property of $|\det|$ function when fixing other columns
$$s \geq   \sup\limits_{(x_{21})_{2} \in  \mathrm{co}F_{2}^{x_{21}}}
|  \det
\left(
\begin{array}{ccc}
T_{0}(x_1+ (x_{10})_1+  (x_{21})_1)  & T_{0}( x_2+ (x_{10})_2  + (x_{21})_2 )  &  (x_{10})_{3}  \\
\end{array}
\right)|.$$
 By (3.5) we may assume $T_0 \mathrm{co}F_{2}^{x_{21}}$ is an ellipsoid  in $\mathbb{R}^{3}$.
Choose a rotation $T_{1} \in O(3)$  such that  $T_1 T_0 \mathrm{co}F_{2}^{x_{21}}$ is an ellipsoid with principal axes parallel to the coordinate axes.
From (3.6) we  may assume  $T_1 T_0 \mathrm{co}F_{2}^{x_{21}}$  is an axis-parallel rectangle.
By the  invariance of $O(3)$, 
\begin{align*}
s &\geq   \sup\limits_{(x_{21})_{2} \in  \mathrm{co}F_{2}^{x_{21}}}
|  \det
\left(
\begin{array}{ccc}
T_{0}(x_1+ (x_{10})_1+  (x_{21})_1)  & T_{0}( x_2+ (x_{10})_2  + (x_{21})_2 )  &  (x_{10})_{3}  \\
\end{array}
\right)| \\
&= \sup\limits_{(x_{21})_{2} \in  \mathrm{co}F_{2}^{x_{21}}}
|  \det
\left(
\begin{array}{ccc}
T_1 T_{0}(x_1+ (x_{10})_1+  (x_{21})_1)  & T_1 T_{0}( x_2+ (x_{10})_2  + (x_{21})_2 )  &  T_1 (x_{10})_{3}  \\
\end{array}
\right)|. 
\end{align*}
Since $T_1 T_0 \mathrm{co}F_{2}^{x_{21}}$  is an axis-parallel rectangle, 
together with
$$\mathcal{S}(T_1 T_{0} \mathrm{co}F_{2}^{x_{21}})=\mathcal{S}(T_1 T_{0} \mathrm{co}F_{2}^{x_{21}}+h), \ 
 \ \forall \ h \in \mathbb{R}^{3}$$
apply inequality (2.6) again to obtain
$$s \geq   \sup\limits_{\substack{  (x_{10})_{3} \in \mathcal{S} (T_0 \mathrm{co}E_{1}^{x_{10}})   \\    (x_{21})_{2} \in \mathcal{S} (T_1 T_0 \mathrm{co}F_{2}^{x_{21}})}}
|  \det
\left(
\begin{array}{ccc}
T_1 T_{0}(x_1+ (x_{10})_1+  (x_{21})_1)  &    (x_{21})_2   &  T_1 (x_{10})_{3}  \\
\end{array}
\right)| $$
holds for any $x_1 \in v(F_3) \subset  \mathfrak{M}^{3 \times 1}$. 

Lastly, applying (2.2) we conclude
\begin{align*}
s &\geq C |T_1 T_{0} v(F_3)+T_1 T_{0} (x_{10})_1+T_1 T_{0}  (x_{21})_1|^{1/3}  |  \mathcal{S} (T_1 T_0 \mathrm{co}F_{2}^{x_{21}})|^{1/3}  |T_1  \mathcal{S} (T_0 \mathrm{co}E_{1}^{x_{10}}) |^{1/3} \\
&= C  |v(F_3)|^{1/3} |\mathrm{co} F_{2}^{x_{21}})|^{1/3}   | \mathrm{co}E_{1}^{x_{10}}) |^{1/3}.
\end{align*}
This together with (3.17) implies
$$s  \geq C |v(F_{3})|^{1/3} |\mathrm{co}F_{2}^{x_{21}}|^{1/3} |\mathrm{co}E_{1}^{x_{10}}|^{1/3}   
\geq C |v(F_{3})|^{1/3} |F_{2}^{x_{21}}|^{1/3} |E_{1}^{x_{10}}|^{1/3}   
\geq C(|E_1||E_2||E_3|)^{1/3}.$$
This  completes (3.1) for $n=3$.

\medskip

For the general $n$,
for each $E_{j}$, denote $F_{j0}= E_{j}$, $1 \leq j \leq n$. 
Given   $1 \leq k \leq n-2$, let 
$$F_{jk}= v(F_{j(k-1)})  \subset  \mathfrak{M}^{n \times (n-k)},  $$
then by (3.4)  there exists fixed $x_{jk}\in v(F_{jk}) \subset  \mathfrak{M}^{n \times (n-k-1)}$, $0 \leq k \leq n-2$, such that
$$|v(F_{jk})| |F_{jk}^{x_{jk}}| \gtrsim  |F_{jk}|=|v(F_{j(k-1)})|.  \eqno (3.18)$$
That is,  for each $E_j$ there exist  $\{x_{j0}, \dots, x_{j(n-2)}\}$ such that  for each $k=0, \dots, n-2$
$$x_{jk} \in v(F_{jk}) \subset  \mathfrak{M}^{n \times (n-k-1)}, $$ 
and
$$|v(F_{j(n-2)})| |F_{j(n-2)}^{x_{j(n-2)}}|   |F_{j(n-3)}^{x_{j(n-3)}}|  \dots  |F_{j1}^{x_{j1}}| |F_{j0}^{x_{j0}}| \gtrsim_{n}  |E_j|.  \eqno (3.19)$$
It is not hard to see
there exist $\{i_j\}_{j=1}^{n}$  with $1 \leq i_{j} \leq n$  and $i_{j} \neq i_{k}$ for $j \neq k$
such that
\begin{align*}
& \ \ \ \  \  (|v(F_{i_{n}(n-2)})| |F_{i_{n-1}(n-2)}^{x_{i_{n-1} (n-2)}}|   |F_{i_{n-2} (n-3)}^{x_{i_{n-2} (n-3)}}|  \dots  |F_{i_{2}1}^{x_{i_{2} 1}}| |F_{i_{1}0}^{x_{i_{1} 0}}|)^{n} \\
& \geq  \prod\limits_{j=1}^{n} (|v(F_{j(n-2)})| |F_{j(n-2)}^{x_{j(n-2)}}|   |F_{j(n-3)}^{x_{j(n-3)}}|  \dots  |F_{j1}^{x_{j1}}| |F_{j0}^{x_{j0}}|)
\gtrsim_{n} \prod\limits_{j=1}^{n}|E_j|.
\end{align*}
For simplicity, denote $i_j=j, 1 \leq j \leq n$. That is,
 $$|v(F_{n(n-2)})| |F_{(n-1)(n-2)}^{x_{(n-1) (n-2)}}|   |F_{(n-2) (n-3)}^{x_{(n-2) (n-3)}}|  \dots  |F_{2 1}^{x_{2 1}}| |F_{10}^{x_{1 0}}| 
\gtrsim_{n} \prod\limits_{j=1}^{n}|E_j|^{1/n}.  \eqno(3.20)$$
To study the suprema, we consider the following $n$-by-$n$ matrices
$$\overline{A}_{1}:=
\left(
\begin{array}{ccc}
(x_{10})_1  &  \dots  &   (x_{10})_n   \\
\end{array}
\right) \in E_{1}$$
with 
$\left(
\begin{array}{ccc}
(x_{10})_1  &  \dots  &   (x_{10})_{(n-1)}   \\
\end{array}
\right)=x_{10} \in \mathfrak{M}^{n \times (n-1)}$
 and 
$(x_{10})_n \in F_{10}^{x_{10}}$;
$$\overline{A}_{2}:=
\left(
\begin{array}{ccc}
(x_{21})_1  &  \dots  &   (x_{21})_n  \\
\end{array}
\right) \in E_{2}$$
with 
$\left(
\begin{array}{ccc}
(x_{21})_1  &  \dots  & (x_{21})_{(n-2)} \\
\end{array}
\right)=x_{21} \in \mathfrak{M}^{n \times (n-2)}$
and
$(x_{21})_{n-1}  \in F_{21}^{x_{21}}$.
That  is,  construct $\{\overline{A}_{1}, \dots, \overline{A}_{n-1}\}$ such that  for each $1 \leq k \leq n-1$
$$\overline{A}_{k}:=
\left(
\begin{array}{ccc}
(x_{k(k-1)})_1 &  \dots  &   (x_{k(k-1)})_n  \\
\end{array}
\right) \in E_{k}, $$
with the condition  that 
$$\left(
\begin{array}{ccc}
x_{k(k-1)})_1 &  \dots  &  (x_{k(k-1)})_{n-k}  \\
\end{array}
\right) =x_{k(k-1)} \in \mathfrak{M}^{n \times (n-k)}, \ \ 
(x_{k(k-1)})_{n-k+1} \in F_{k(k-1)}^{x_{k(k-1)}}.$$
For any $\overline{A}_{n}:=
\left(
\begin{array}{ccc}
x_1 &  \dots  &  x_n  \\
\end{array}
\right) \in E_{n}$, 
 for any constructed $\overline{A}_{1}, \dots, \overline{A}_{n-1}$ above,
\begin{align*}
s &\geq | \det (\overline{A}_{1} + \dots +  \overline{A}_{n-1}+ \overline{A}_{n}) | \\
&=|  \det
\left(
\begin{array}{ccc}
x_1+\sum\limits_{k=1}^{n-1}  (x_{k(k-1)})_1 &  \dots  &  x_n+\sum\limits_{k=1}^{n-1}  (x_{k(k-1)})_n   \\
\end{array}
\right)|.
\end{align*}
Taking the same arguments as in the case $n=3$,  there exist $T_0, T_1 \in O(n)$
$$s \geq \sup\limits_{\substack{  (x_{10})_n \in  \mathcal{S}(T_0 \mathrm{co}F_{10}^{x_{10}})    \\   (x_{21})_{(n-1)} \in  \mathcal{S}(T_1 T_0 \mathrm{co}F_{21}^{x_{21}} )}}
| \det
\left(
\begin{array}{cc}
B &  B^{\prime} \\
\end{array}
\right)|,  \eqno(3.21)$$
where 
$$B =  T_1 T_0
\left(
\begin{array}{ccc}
x_1+\sum\limits_{k=1}^{n-1}  (x_{k(k-1)})_1   &  \dots  &  x_{n-2}+\sum\limits_{k=1}^{n-1}  (x_{k(k-1)})_{n-2}   \\
\end{array}
\right) \in \mathfrak{M}^{n \times (n-2)},$$
$$B^{\prime} = 
\left(
\begin{array}{cc}
 (x_{21})_{(n-1)}   &   T_1 (x_{10})_n    \\
\end{array}
\right) \in \mathfrak{M}^{n \times 2}.$$
Applying the same arguments again to (3.21),  there exist $T_2 \in O(n)$ 
$$s \geq \sup\limits_{\substack{  (x_{10})_n \in  \mathcal{S}(T_0 \mathrm{co}F_{10}^{x_{10}})    \\   (x_{21})_{(n-1)} \in  \mathcal{S}(T_1 T_0 \mathrm{co}F_{21}^{x_{21}} )
\\ (x_{32})_{n-2} \in  \mathcal{S}(T_2 T_1 T_0 \mathrm{co}F_{32}^{x_{32}}) }}
| \det
\left(
\begin{array}{cc}
C &  C^{\prime} \\
\end{array}
\right)|,  \eqno(3.22)$$
where 
$$C =  T_2 T_1 T_0
\left(
\begin{array}{ccc}
x_1+\sum\limits_{k=1}^{n-1}  (x_{k(k-1)})_1   &  \dots  &  x_{n-3}+\sum\limits_{k=1}^{n-1}  (x_{k(k-1)})_{n-3}   \\
\end{array}
\right) \in \mathfrak{M}^{n \times (n-3)},$$
$$C^{\prime} = 
\left(
\begin{array}{ccc}
(x_{32})_{(n-2)}   &   T_2 (x_{21})_{(n-1)}   &   T_2 T_1 (x_{10})_n    \\
\end{array}
\right) \in \mathfrak{M}^{n \times 3}.$$
Keep repeating the  same arguments above and  finally we have there exists $T_0, \dots, T_{n-2} \in O(n)$,  
such that  for any $x_1 \in v(F_{n(n-2)}) \subset \mathfrak{M}^{1}$
$$s \geq \sup\limits_{\substack{  (x_{10})_n \in  \mathcal{S}(T_0 \mathrm{co}F_{10}^{x_{10}})    \\   (x_{21})_{(n-1)} \in  \mathcal{S}(T_1 T_0 \mathrm{co}F_{21}^{x_{21}} )
\\ \dots \dots \\
(x_{(n-1)(n-2)})_{2} \in  \mathcal{S}(T_{n-2}  T_{n-3} \dots T_0 \mathrm{co}F_{(n-1)(n-2)}^{x_{(n-1)(n-2)}} )}}
| \det
\left(
\begin{array}{cc}
D &  D^{\prime} \\
\end{array}
\right)|,  \eqno(3.23)$$
where $D \in  \mathfrak{M}^{n \times 1}$, $D^{\prime} \in \mathfrak{M}^{n \times (n-1)}$:
$$D= (T_{n-2} \dots T_0) (x_1 + \sum\limits_{k=1}^{n-1}   (x_{k(k-1)})_1),$$
$$D^{\prime}=\left(
\begin{array}{cccccc}
 (x_{(n-1)(n-2)})_2  &   T_{n-2}(x_{(n-2)(n-3)})_3  &     (T_{n-2} T_{n-3}) (x_{(n-3)(n-4)})_4    &  \dots &   (T_{n-2} \dots T_1) (x_{10})_n  \\
\end{array}
\right).$$

It follows from  (2.2) together with the invariance under $O(n)$ that
$$s\geq C |v(F_{n(n-2)})|^{1/n}  |\mathrm{co}F_{(n-1)(n-2)}^{x_{n-1)(n-2)}}|^{1/n} 
| \mathrm{co}F_{(n-2)(n-3)}^{x_{(n-2)(n-3)}}|^{1/n}  \dots  |\mathrm{co}F_{21}^{x_{21}}|^{1/n} |\mathrm{co}F_{10}^{x_{10}}|^{1/n}.$$ 
Obviously, 
$$|\mathrm{co}F_{k(k-1)}^{x_{k(k-1)}}| \geq  |F_{k(k-1)}^{x_{k(k-1)}}|, \  1 \leq k \leq n-1.$$
This together with (3.20) implies
\begin{align*}
s &\geq C(|v(F_{n(n-2)})| | \mathrm{co}F_{(n-1)(n-2)}^{x_{(n-1) (n-2)}}|   | \mathrm{co}F_{(n-2) (n-3)}^{x_{(n-2) (n-3)}}|
\dots  | \mathrm{co}F_{2 1}^{x_{2 1}}| |F_{10}^{x_{1 0}}|)^{1/n}  \\
&\geq C (|v(F_{n(n-2)})| |F_{(n-1)(n-2)}^{x_{(n-1) (n-2)}}|   |F_{(n-2) (n-3)}^{x_{(n-2) (n-3)}}|  \dots  |F_{2 1}^{x_{2 1}}| |F_{10}^{x_{1 0}}|  )^{1/n}
\geq C \prod\limits_{j=1}^{n}|E_j|^{\frac{1}{n^{2}}}.  
\end{align*}
This completes Theorem 3.1.

\end{proof}

\medskip

\begin{corollary}
There exists a finite constant $\mathcal{A}_{n}, \mathcal{B}_{n}$ such that  for any  measurable set $E \subset \mathfrak{M}^{n \times n}$ of finite measure,
for any non-zero scalar $\lambda_{j} \in \mathbb{R}$, $j=1, \dots, n$,
$$(\prod_{j=1}^{n} |\lambda_{j}|) |E|^{\frac{1}{n}}  
\leq \mathcal{A}_{n} \displaystyle{\sup_{\substack{A_{j} \in E  \\  j=1, \dots, n}}}  \  | \det(\lambda_{1} A_{1}+ \dots + \lambda_{n} A_{n} )  |.  \eqno (3.24)$$
If $E$ is a compact convex set in $\mathfrak{M}^{n \times n}$, then 
$$|E|^{\frac{1}{n}}  \leq \mathcal{B}_{n} \displaystyle{\sup_{A \in E}}  \  | \det(A)  |.  \eqno (3.25)  $$

\end{corollary}

\begin{proof}
To see (3.25), let $E_{j}=\lambda_{j}E$.
Applying  Theorem 3.1 gives
$$\prod_{j=1}^{n}  |\lambda_{j} E|^{\frac{1}{n^{2}}}  
\leq \mathcal{C}_{n} \displaystyle{\sup_{\substack{A_{j} \in E  \\  j=1, \dots, n}}}  \  | \det(\lambda_{1} A_{1}+ \dots + \lambda_{n} A_{n} ) |,$$
which  implies (3.25).
In particular, if $E \subset \mathfrak{M}^{n \times n}$ is a compact convex set,
setting $\lambda_{j}=\frac{1}{n}$, $j=1, \dots, n$, it follows from (3.24) that
$$ (\frac{1}{n})^{n} |E|^{1/n}  \leq   \mathcal{A}_{n}    \displaystyle{\sup_{\substack{A_{j} \in E \\ j=1, \dots, n}}}   | \det(\frac{1}{n}A_{1}+ \dots +\frac{1}{n} A_{n} )  |.$$
 On the other hand, since $E$ is convex,
$$\displaystyle{\sup_{A\in E}}   | \det(A)  |
\geq \displaystyle{\sup_{\substack{A_{j} \in E \\ j=1, \dots, n} }}   | \det(\frac{1}{n}A_{1}+ \dots +\frac{1}{n} A_{n} )  |.$$
Thus we get (3.25).

\end{proof}

\medskip

Here we give a direct way to see  Lemma 13.2 \cite{Christ} which follows from (3.25).
Let $E \subset \mathfrak{M}^{n \times n}$ be a  measurable set.
The inequality (1.18) in  Lemma 13.2  has  translation invariance property,
so we  assume that $0 \in E$.
Given any  matrices $A_{1}, \dots, A_{n^{2}}$ in $E$,  from (3.25) it follows that
$$ |\mathrm{co} \{0,A_{1}, \dots, A_{n^{2}} \} |^{\frac{1}{n}}  \lesssim_{n}  \sup\limits_{A \in \mathrm{co} \{0,A_{1}, \dots, A_{n^{2}} \} }  \  | \det(A)  |,   \eqno (3.26) $$
By (2.2),
there   exist $A_{1}, \dots, A_{n^{2}}$ such that
$$|E|  \lesssim_{n} | \mathrm{co} \{0,A_{1}, \dots, A_{n^{2}} \} |,$$
together with (3.26) we obtain that
$$|E|^{\frac{1}{n}} \lesssim_{n}     \sup\limits_{A \in \mathrm{co} \{0,A_{1}, \dots, A_{n^{2}} \} }  \  | \det(A)  |.   \eqno (3.27)  $$
For any convex set $F \subset \mathfrak{M}^{n \times n}$
$$ \sup\limits_{A \in \mathrm{co} \{0, F \}}  \  | \det(A)|=\sup\limits_{A \in F}  \  | \det(A)|,$$
since $ |\det( \lambda A)|= \lambda^{n} |\det( A)| \leq | \det(A)|$  for any  $\lambda \in [0,1]$.
So
$$\sup\limits_{A \in \mathrm{co} \{0,A_{1}, \dots, A_{n^{2}} \} }  \  | \det(A)  |
= \sup\limits_{A \in \mathrm{co} \{A_{1}, \dots, A_{n^{2}} \} }  \  | \det(A)  |.   \eqno (3.28)$$
Denote $A^{(k)}$ by the $k$-th column vector of the matrix $A$, $1 \leq k \leq n$ . Then there exist
$\widetilde{A}_{1}, \dots, \widetilde{A}_{n} \in \{A_{1}, \dots, A_{n^{2}}\}$  \
($\widetilde{A}_{i}$,  $\widetilde{A}_{j}$ might be the same matrix),  such that
for any $ \{\lambda_{1}, \dots,\lambda_{n^{2}} \}$  satisfying
$\sum\limits_{j=1}^{n^{2}} \lambda_{j} =1$ and $0 \leq \lambda_{j} \leq 1$,
$$| \det( \lambda_{1}A_{1}+ \dots+ \lambda_{n^{2}}A_{n^{2}} )  |
\leq  | \sum\limits_{ \substack{i_{j} \in \{1, \dots, n \} \\  i_{j} \neq  i_{k}, \forall j \neq  k }}  \det ( \widetilde{A}_{i_{1}}^{(1)},  \dots,  \widetilde{A}_{i_{n}}^{(n)} ) |   \eqno (3.29)$$
holds, this is because
$$\sum\limits_{1 \leq l_{1}, \dots, l_{n} \leq n^{2}} \lambda_{l_{1}} \dots \lambda_{l_{n}}  
\leq  \sum\limits_{1 \leq l_{1}, \dots, l_{n-1} \leq n^{2}} \lambda_{l_{1}} \dots \lambda_{l_{n-1}}
\leq \dots \leq \sum\limits_{1 \leq l_{1},  l_{2} \leq n^{2}} \lambda_{l_{1}} \lambda_{l_{2}}
\leq \sum\limits_{1 \leq l_{1}  \leq n^{2}} \lambda_{l_{1}}= 1.$$
Hence from (3.27)-(3.29)
$$|E|^{\frac{1}{n}}  \lesssim_{n}   | \sum\limits_{ \substack{i_{j} \in \{1, \dots, n \} \\  i_{j} \neq  i_{k}, \forall j \neq  k }}  \det ( \widetilde{A}_{i_{1}}^{(1)},  \dots,  \widetilde{A}_{i_{n}}^{(n)} ) |.    \eqno (3.30)$$
As mentioned in the proof of Lemma 13.2 \cite{Christ},
$ \sum\limits_{ \substack{i_{j} \in \{1, \dots, n \} \\  i_{j} \neq  i_{k}, \forall j \neq  k }}  \det ( \widetilde{A}_{i_{1}}^{(1)},  \dots,  \widetilde{A}_{i_{n}}^{(n)} ) $  is
$\mathbb{Z}$-linear combination of
$\{\det ( \sum\limits_{j=1}^{n}  s_{j} \widetilde{A}_{j} ): s_{j} \in \{0,1\}  \}$.
This gives (1.20):
$$ |E|^{\frac{1}{n}}
\lesssim_{n}  \sup\limits_{\substack{A_{1}, \dots, A_{n} \in E  \\  s_{1}, \dots, s_{n} \in \{0, 1\} }} |\det( s_{1}A_{1}+ \dots + s_{n}A_{n})|.$$


Obviously, (3.25) is not affine invariant.
The following example shows balls or ellipsoids are not the optimisers.

\hspace{-12pt}{\bf Example 3.2.}
(i)\ Let $n=2$, $E=B(0,r)$,
$A=
\left(
 \begin{array}{cc}
   a & c  \\
   b & d  \\
 \end{array}
\right) \in E$. \\
Then $\sup\limits_{A\in E} |\det(A)|=\frac{r^{2}}{2}$ by calculation.
Consider the ellipsoid $F$ in $\mathbb{R}^{4}$ with $|F|=|B(0,r)|$,
$$F= \{
\left(
 \begin{array}{cc}
   a & c  \\
   b & d  \\
 \end{array}
\right): \frac{a^{2}}{l_{1}^{2}}+ \frac{b^{2}}{l_{2}^{2}}+\frac{c^{2}}{l_{3}^{2}}+\frac{d^{2}}{l_{4}^{2}} \leq 1 \}.$$
It is easy to obtain
$\sup\limits_{A\in F} |\det(A)| \geq \frac{l_{1}l_{4}+l_{2}l_{3}}{4} \geq \frac{r^{2}}{2}$ by GM-AM inequality.

(ii)\ Let $r=1$. Since $A\mapsto |\det(A)|$ is a continuous function on $E=B(0,1)$ under the natural topology on Euclidean space $\mathbb{R}^{4}$,
there exists $0<\delta<\frac{1}{25}$ such that $|\det(A)|\leq \frac{1}{4}$ for all $A\in E$ satisfying
\begin{align*}
|A-
\left(
 \begin{array}{cc}
   1 & 0  \\
   0 & 0  \\
 \end{array}
\right)|
= (a-1)^{2}+b^{2}+c^{2}+d^{2}
\leq 2\delta.
\end{align*}
Then  for all  $A\in E$ satisfying $\sqrt{1-\delta} \leq a \leq 1$, we have
$$b^{2}+c^{2}+d^{2} \leq 1-a^{2} \leq 1- (1-\delta) = \delta.$$
Thus 
\begin{align*}
|A-
\left(
 \begin{array}{cc}
   1 & 0  \\
   0 & 0  \\
 \end{array}
\right)|
= (a-1)^{2}+b^{2}+c^{2}+d^{2}
\leq (1- \sqrt{1-\delta})^{2} + \delta
\leq 2\delta
\end{align*}
which implies that
$|\det(A)|\leq \frac{1}{4}$ for any $A\in E$ satisfying $\sqrt{1-\delta} \leq a \leq 1$.

Let
$P=
\left(
 \begin{array}{cc}
   0 & 0  \\
   0 & p  \\
 \end{array}
\right)$
with $p=\frac{1}{\sqrt{1-\delta}}$ and then consider $\sup\limits_{A\in \mathrm{co} \{P\cup E\}} |\det(A)|$,
\begin{align*}
\ \ \ \ \ \sup\limits_{A\in \mathrm{co} \{P\cup E\}} |\det(A)|
&=\sup\limits_{A\in E, \lambda \in [0,1]} |\det(\lambda A+(1-\lambda)P) |\\
&=  \sup\limits_{A\in E, \lambda \in [0,1]}
|\det
\left(
 \begin{array}{cc}
  \lambda a &  \lambda c \\
  \lambda b &  \lambda d+ (1-\lambda)p \\
 \end{array}
\right)
|\\
&= \sup\limits_{A\in E, \lambda \in [0,1]}
|\det
\left(
 \begin{array}{cc}
  \lambda a &  \lambda c \\
  \lambda b &  \lambda d \\
 \end{array}
\right)+
\det
\left(
 \begin{array}{cc}
  \lambda a &  0 \\
  \lambda b & (1-\lambda)p \\
 \end{array}
\right)
|.
\end{align*}

When $a\not\in [\sqrt{1-\delta}, 1]$,
\begin{align*}
&\  \sup\limits_{A\in E, \lambda \in [0,1]}
|\det
\left(
 \begin{array}{cc}
  \lambda a &  \lambda c \\
  \lambda b &  \lambda d \\
 \end{array}
\right)+
\det
\left(
 \begin{array}{cc}
  \lambda a &  0 \\
  \lambda b & (1-\lambda)p \\
 \end{array}
\right)
| \\
&\leq \sup\limits_{\lambda \in [0,1]}
\lambda^{2} \frac{1}{2} +  \lambda(1-\lambda) a p \\
&\leq \sup\limits_{\lambda \in [0,1]}
\lambda^{2} \frac{1}{2} +  \lambda(1-\lambda) \sqrt{1-\delta} \frac{1}{\sqrt{1-\delta}} \\
&= \sup\limits_{\lambda \in [0,1]}
\lambda^{2} \frac{1}{2} +  \lambda(1-\lambda) \leq  \frac{1}{2}.
\end{align*}

When $a \in [\sqrt{1-\delta}, 1]$,
\begin{align*}
&\  \sup\limits_{A\in E, \lambda \in [0,1]}
|\det
\left(
 \begin{array}{cc}
  \lambda a &  \lambda c \\
  \lambda b &  \lambda d \\
 \end{array}
\right)+
\det
\left(
 \begin{array}{cc}
  \lambda a &  0 \\
  \lambda b & (1-\lambda)p \\
 \end{array}
\right)
| \\
&\leq \sup\limits_{\lambda \in [0,1]}
\lambda^{2} \frac{1}{4} +  \lambda(1-\lambda)  p \\
&=\sup\limits_{\lambda \in [0,1]}
\lambda^{2} \frac{1}{4} +  \lambda(1-\lambda)   \frac{1}{\sqrt{1-\delta}}.
\end{align*}
It is easy to see for $0<\delta<\frac{1}{25}$ given above,
$$\sup\limits_{\lambda \in [0,1]}
\lambda^{2} \frac{1}{4} +  \lambda(1-\lambda)   \frac{1}{\sqrt{1-\delta}} \leq  \frac{1}{2}.$$
Therefore, 
$$\sup\limits_{A\in \mathrm{co} \{P\cup E\}} |\det(A)| = \sup\limits_{A\in E} |\det(A)|,$$
which implies balls can not be the optimisers.

\medskip

\hspace{-12pt}{\bf Remark 3.3.}
Let $E \subset \mathfrak{M}^{n \times n}$ be a compact convex set.              
 If we compare the maximal volume of simiplicies
$\sup\limits_{A_{0}, \dots, A_{n^{2}} \in E} \mathrm{vol}( \mathrm{co} \{ A_{0}, \dots, A_{n^{2}} \})$
contained in $E$ with the $\displaystyle{\sup_{A \in E}}  \  | \det(A)  |$,
it follows from (3.25) that 
$$\sup\limits_{A_{0}, \dots, A_{n^{2}} \in E} \mathrm{vol}( \mathrm{co} \{A_{0}, \dots, A_{n^{2}} \})
\lesssim_{n}  \displaystyle{\sup_{A \in E}}  \  | \det(A)  |^{n}.  \eqno (3.31)$$
Indeed by John ellipsoids, it is enough to consider the case when $E$ is a ellipsoid in  $\mathfrak{M}^{n \times n}$.
For any ellpsoid 
$$E\equiv \{x\in \mathbb{R}^{n^{2}}:  \displaystyle{\sum_{i}^{n^{2}}} \frac{ | \langle x-x_{0}, \omega_{i} \rangle |^{2}}{l_{i}^{2}} \leq 1  \},$$
where $x_{0}\in \mathbb{R}^{n^{2}}$, $\{\omega_{i}\}$ is an orthonormal basis in $\mathbb{R}^{n^{2}}$.
By the affine invariance of $\sup\limits_{A_{0}, \dots, A_{n^{2}} \in E} \mathrm{vol}( \mathrm{co} \{ A_{0}, \dots, A_{n^{2}} \})$,
it is enough to see balls centred at $0$. 
Apply the Hadamard inequality, for any $A_{j}  \in B(0,r) \subset \mathbb{R}^{n^{2}}$, $j=0, \dots, n^{2}$
$$\mathrm{vol}( \mathrm{co} \{ A_{0}, \dots, A_{n^{2}} \}) 
\leq  |A_0-A_1 |  |A_0-A_2| \dots  |A_0-A_{n^{2}}| \lesssim_{n} r^{n^{2}} \sim |B(0,r)| .$$
Hence for any ellipsoid $E   \subset \mathbb{R}^{n^{2}}$,
$$\sup\limits_{A_{0}, \dots, A_{n^{2}} \in E} \mathrm{vol}( \mathrm{co} \{A_{0}, \dots, A_{n^{2}} \})
\lesssim_{n}  |E|.$$
On the other hand,  by (3.25)
$$|E| \lesssim_{n} \displaystyle{\sup_{A \in E}}  \  | \det(A)  |^{n}.$$
Therefore, we have  the following relation
$$\sup\limits_{A_{0}, \dots, A_{n^{2}} \in E} \mathrm{vol}( \mathrm{co} \{ A_{0}, \dots, A_{n^{2}} \})
\lesssim_{n}  \displaystyle{\sup_{A \in E}}  \  | \det(A)  |^{n}.$$
Similarly,  we have
$$\sup\limits_{A_{1}, \dots, A_{n^{2}} \in E} \mathrm{vol}( \mathrm{co} \{ 0, A_{1}, \dots, A_{n^{2}} \})
\lesssim_{n}  \displaystyle{\sup_{A \in E}}  \  | \det(A)  |^{n}.  \eqno (3.32)$$
If $0 \in E$, it is true which mainly due to the Hadamard inequality and the 
$\mathrm{GL}_{n}(\mathbb{R})$  invariance of  
$\sup\limits_{A_{1}, \dots, A_{n^{2}} \in E} \mathrm{vol}( \mathrm{co} \{ 0, A_{1}, \dots, A_{n^{2}} \})$. 
If $0 \not \in E$, the relation above still holds because of the fact
$$\sup\limits_{A \in E}  | \det(A)  |^{n}= \sup\limits_{A \in \mathrm{co}\{0, E\}}  | \det(A)  |^{n}.$$

\bigskip

\hspace{-13pt}{\bf Acknowledgments.}\quad

I am  grateful to my supervisor Professor Carbery for his helpful suggestions and revision on this paper.
This work was supported by the scholarship from China Scholarship Council.

\bigskip

\bibliographystyle{amsalpha}

\end{document}